\documentclass[
10pt
]{article}

\usepackage{color}
\definecolor{LinkColor}{rgb}{0,0,1}
\definecolor{lbcolor}{rgb}{0.85,0.85,0.85}
\definecolor{FrameColor}{rgb}{0.85,0.85,0.85}

\usepackage[ngerman, french, english]{babel}
\usepackage[utf8]{inputenc}
\usepackage[T1]{fontenc}
\usepackage{enumerate}
\usepackage{setspace}

\def\pskip{\\[-3mm]}
\def\bpskip{\\[-2mm]}

\usepackage[babel,french=guillemets,german=swiss]{csquotes}

\usepackage{amsmath}
\usepackage{amssymb}
\usepackage{dsfont}
\usepackage{nicefrac}
\usepackage{empheq}
\allowdisplaybreaks

\usepackage[%
pdftitle={Titel},
pdfauthor={Autor},
pdfcreator={LaTeX, LaTeX with hyperref and KOMA-Script},
pdfsubject={Betreff}, 
pdfkeywords={Keywords}
]{hyperref} 

\hypersetup{colorlinks=true,
	linkcolor=LinkColor,%
	anchorcolor=LinkColor,%
	citecolor=LinkColor,%
	filecolor=LinkColor,%
	menucolor=LinkColor,%
	urlcolor=LinkColor,%
	allcolors=LinkColor}

\usepackage{amsthm}

\newtheoremstyle{tstyle}
{15pt}	
{5pt}	
{\itshape}	
{}	
{\bfseries}	
{.}	
{0.5em}	
{}	

\theoremstyle{tstyle}

\newtheorem{thm}{Theorem}
\newtheorem{lem}[thm]{Lemma}
\newtheorem{prop}[thm]{Proposition}
\newtheorem{cor}[thm]{Corollary}
\newtheorem{defn}[thm]{Definition}

\newtheoremstyle{cstyle}
{15pt}	
{5pt}	
{}	
{}	
{\bfseries}	
{}	
{0.2222em}	
{}	

\theoremstyle{cstyle}

\newtheorem*{com}{Comment}

\makeatletter
\g@addto@macro{\thm@space@setup}{\thm@headpunct{}}
\renewenvironment{proof}[1][\proofname]{\par
	\pushQED{\qed}%
	\normalfont \topsep6\p@\@plus6\p@\relax
	\trivlist
	\item[\hskip\labelsep
	\bfseries
	#1\@addpunct{\,}]\ignorespaces
}{%
	\popQED\endtrivlist\@endpefalse
}
\makeatother

\makeatletter
\g@addto@macro{\thm@space@setup}{\thm@headpunct{}}
\newenvironment{sketch-proof}[1][Sketch of the proof]{\par
	\pushQED{\qed}%
	\normalfont \topsep6\p@\@plus6\p@\relax
	\trivlist
	\item[\hskip\labelsep
	\bfseries
	#1\@addpunct{\,}]\ignorespaces
}{%
	\popQED\endtrivlist\@endpefalse
}
\makeatother

\def\RR{\mathbb R}
\def\WW{{\mathcal W}}
\def\HH{{\mathcal H}}
\def\VV{{\mathcal V}}
\def\NN{\mathbb N}

\def\BB{{\mathbb B_K}}

\def\MM{\mathbb M}

\def\eps{\varepsilon}
\def\supp{\textnormal{supp }}

\def\BR{{B_R(0)}}

\def\K{{K}}

\def\ddt{\frac{\mathrm d}{\mathrm dt}}

\def\ddtau{\frac{\mathrm d}{\mathrm d\tau}}

\def\dtau{\;\mathrm d\tau}
\def\dz{\;\mathrm dz}
\def\dx{\;\mathrm dx}
\def\dy{\;\mathrm dy}
\def\dv{\;\mathrm dv}
\def\dt{\;\mathrm dt}
\def\ds{\;\mathrm ds}

\def\B{{\bar B}}

\def\delzi{\partial_{z_i}}
\def\delzj{\partial_{z_j}}
\def\delxi{\partial_{x_i}}
\def\delxj{\partial_{x_j}}

\def\delt{\partial_{t}}
\def\delx{\partial_{x}}
\def\delv{\partial_{v}}
\def\delz{\partial_{z}}

\def\grad{\nabla}
\def\laplace{\Delta}

\def\bigvert{\;\big\vert\;}
\def\Bigvert{\;\Big\vert\;}

\def\lbr{\left\{ }
\def\rbr{\right\} }
\def\conf{\hspace{0.1cm}\widehat{=}\hspace{0.1cm}}

\def\tand{\quad\text{and}\quad}
\def\twith{\quad\text{with}\quad}

\def\B*{{\bar B}}
\def\u*{{\bar u}}
\def\f*{f_{\u*}}
\def\g*{g_{\u*}}

\def\wto{\rightharpoonup}

\newcommand{\Underset}[3][0pt]{\ensuremath{\underset{\raise#1\hbox{\small\ensuremath{#2}}}{#3}}}

\begin{document}

\begin{center}	
\LARGE{Optimal control of a Vlasov-Poisson plasma by an external magnetic field \\[2mm] The basics for variational calculus}\\[5mm]
\normalsize{P. Knopf}\\
\textit{University of Bayreuth, 95440 Bayreuth, Germany}
\texttt{Patrik.Knopf@uni-bayreuth.de}
\end{center}

\begin{abstract}
	We consider the three dimensional Vlasov-Poisson system in the plasma physical case. It describes the time evolution of the distribution function of a large number of electrically charged particles which move under the influence of a self-consistent electric field that is generated by the charge of the particles. Thereby the particles obey Newton's laws of motion. Collisions of the particles, electrodynamic and relativistic effects are neglected in this model. The aim of various concrete applications is to control the distribution function of the plasma in a desired way by an external magnetic field that interacts with the particles via Lorentz force. This can be modeled by an optimal control problem. For that reason the basics for calculus of variations will be introduced in this paper. We have to find a suitable class of fields that are admissible for this procedure as they provide unique global solutions of the Vlasov-Poisson system. Then we can define a field-state operator that maps any admissible field onto its corresponding distribution function. We will show that this field-state operator is Lipschitz continuous and (weakly) compact. Last we will consider a model problem with a tracking type cost functional and we will show that this optimal control problem has at least one globally optimal solution.\\
	
	\textit{Keywords}: Vlasov-Poisson equation, optimal control, nonlinear partial differential equations, calculus of variations.
\end{abstract}

\section{Introduction}

The three dimensional Vlasov-Poisson system in the plasma physical case is given by the following system of partial differential equations:
\begin{equation}
\label{VP}
\begin{cases}
\partial_t f + v\cdot \partial_x f - \partial_x \psi \cdot \partial_v f = 0,\\[0.25cm]
- \Delta \psi = 4\pi \rho, \quad \lim_{|x|\to\infty} \psi(t,x) = 0,\\[0.25cm]
\rho(t,x) = \int f(t,x,v)\ \mathrm dv.
\end{cases}
\end{equation}
Here $f=f(t,x,v)\ge 0$ denotes the distribution function of the particle ensemble that is a scalar function representing the density in phase space. Its time evolution is described by the first line of \eqref{VP} which is a first order partial differential equation that is referred to as the Vlasov equation. For any measurable set $M\subset \mathbb R^6$,
$\int_M f(t,x,v)\, \mathrm d(x,v)$
yields the charge of the particles that have space coordinates $x\in\mathbb R^3$ and velocity coordinates $v\in\mathbb R^3$ with $(x,v)\in M$ at time $t\ge 0$. The function $\psi$ is the electrostatic potential that is induced by the charge of the particles. It is given by Poisson's equation $-\laplace \psi = 4\pi\rho$ with an homogeneous boundary condition where $\rho$ denotes the volume charge density. The self-consistent electric field is then given by $-\partial_x \psi$. Note that both $\psi$ and $-\partial_x \psi$ depend linearly on $f$. Hence the Vlasov-Poisson system is nonlinear due to the term $-\partial_x \psi \cdot \partial_v f$ in the Vlasov equation. Assuming $f$ to be sufficiently regular (e.g., $f(t) := f(t,\cdot,\cdot)\in C^1_c (\mathbb R^6)$ for all $t \ge 0$), we can solve Poisson's equation explicitly and obtain
\begin{equation}
\label{PSIF}
\psi_f(t,x) = \iint \frac{f(t,y,w)}{|x-y|} \;\mathrm dw \mathrm dy\ \text{  for } t\ge 0, x\in\mathbb R^3.
\end{equation}
Considering $f\mapsto \psi_f$ to be a linear operator we can formally rewrite the Vlasov-Poisson system as
\begin{equation}
\label{VP2}
\partial_t f + v\cdot \partial_x f - \partial_x \psi_f \cdot \partial_v f = 0.
\end{equation}
Combined with the condition
\begin{equation}
\label{IC}
f|_{t=0} = \mathring f
\end{equation}
for some function $\mathring f\in C^1_c(\RR^6)$ we obtain an initial value problem. A first local existence and uniqueness result to this initial value problem was proved by R. Kurth \cite{kurth}. Later J. Batt \cite{batt} established a continuation criterion which claims that a local solution can be extended as long as its velocity support is under control. Finally, two different proofs for global existence of classical solutions were established independently and almost simultaneously, one by K. Pfaffelmoser \cite{pfaffelmoser} and one by P.-L. Lions and B. Perthame \cite{lions-perthame}. Later, a greatly simplified version of Pfaffelmoser's proof was published by J. Schaeffer \cite{schaeffer}. This means that the follwing result is established: Any nonnegative initial datum \linebreak$\mathring f \in C^1_c(\mathbb R^6)$ launches a global classical solution $f\in C^1([0,\infty[\times\mathbb R^6)$ of the Vlasov-Poisson system \eqref{VP} satisfying the initial condition \eqref{IC}. Moreover, for every time $t\in[0,\infty[$, $f(t)=f(t,\cdot,\cdot)$ is compactly supported in $\mathbb R^6$.
Hence equation \eqref{PSIF} and the reformulation of the Vlasov-Poisson system \eqref{VP2} are well-defined in the case $\mathring f \in C_c^1(\mathbb R^6)$. For more information we recommend to consider the article \cite{rein} by G. Rein that gives an overview on the most important results.\pskip

To control the distribution function $f$ we will add an external magnetic field $B$ to the Vlasov equation:
\begin{equation}
\label{VPC}
\partial_t f + v\cdot \partial_x f - \partial_x \psi_f \cdot \partial_v f + (v\times B) \cdot \partial_v f = 0, \quad f|_{t=0} = \mathring f\;.
\end{equation}
The cross product $v\times B$ occurs since, unlike the electric field, the magnetic field interacts with the particles via Lorentz force.  If we want to discuss an optimal control problem where the PDE-constraint is given by \eqref{VPC} we must firstly establish the basics for variational calculus:
\begin{itemize}
	\item We will introduce a set $\BB$ such that any field $B\in\BB$ induces a \textit{unique} and \textit{sufficiently regular} strong solution $f=f_B$ of the initial value problem~\eqref{VPC} that exists on any given time interval $[0,T]$. The set $\BB$ will be referred to as the \textit{set of admissible fields}.
	\item We will show that the solution $f_B$ depends Lipschitz-continuously on the field $B$ while the partial derivatives $\delt f_B$, $\delx f_B$ and $\delv f_B$ depend Hölder-continuously on $B$.
	\item We will also prove that the operator $B\mapsto f_B$ is \textit{compact/weakly compact} in some suitable sense.
\end{itemize}
Finally, we will use these results to discuss a model problem of optimal control: Let $\mathring f \in C^2_c(\RR^6)$ be any given initial datum and let $T>0$ be any given final time. The aim is to control the distribution function $f_B$ in such a way that its value at time $T$ matches a desired distribution $f_d$ as closely as possible. We will consider the following optimal control problem
\begin{align*}
\text{Minimize}\; J(B) := \frac 1 2 \|f_B(T) - f_d\|_{L^2(\RR^6)}^2 + \frac \lambda 2 \|D_x B\|_{L^2(]0,T[\times\RR^3;\RR^3)}, \;\, \text{s.t.}\; B\in\BB
\end{align*}
and we will show that this problem has at least one optimal solution.

\section{Notation and preliminaries}

Our notation is mostly standard or self-explaining. However, to avoid misunderstandings, we fix some of it here. We will also present some basic results that are necessary for the later approach.

Let $d\in\NN$, $U\subset \RR^d$ be any open subset, $k\in \NN$ and $1\le p\le \infty$ be arbitrary. $C^k(U)$ denotes the space of $k$ times continuously differentiable functions on $U$, $C_c(U)$ denotes the space of $C^k(U)$-functions having compact support in $U$ and $C^k_b(U)$ denotes the space of $C^k(U)$-functions that are bounded with respect to the norm
\begin{align*}
\|u\|_{C^k_b(U)} := \underset{|\alpha|\le k}{\sup} \; \|D^\alpha u\|_\infty = \underset{|\alpha|\le k}{\sup} \; \underset{x\in U}{\sup}\; |D^\alpha u(x)| , \quad u\in C^k(U).
\end{align*}
For any $\gamma \in ]0,1]$, $C^{k,\gamma}(U)$ denotes the space of Hölder-continuous $C^k(U)$-functions, i.e.,
\begin{align*}
C^{k,\gamma}(U) := \big\{ u\in C^k(U) \bigvert \|u\|_{C^{k,\gamma}(U)} < \infty \big\}
\end{align*}
where for any $u\in C^k(U)$,
\begin{align*}
\|u\|_{C^k_b(U)} := \underset{|\alpha|\le k}{\sup} \big\{ \|D^\alpha u\|_\infty, [D^\alpha u]_\gamma \big\} \;\;\text{with}\;\; [D^\alpha u]_\gamma:= \underset{x\neq y}{\sup} \frac{|D^\alpha u(x) - D^\alpha u(y)|}{|x-y|^\gamma}.
\end{align*}
Note that $\big( C^k_b(U), \|\cdot\|_{C^k_b(U)} \big)$ and $\big( C^{k,\gamma}(U), \|\cdot\|_{C^{k,\gamma}(U)} \big)$ are Banach spaces. $L^p(U)$ denotes the standard $L^p$-space on $U$ and $W^{k,p}(U)$ denotes the standard Sobolov space on $U$ as, for instance, defined by E. Lieb and M. Loss in \linebreak\cite[s.\,2.1,6.7]{lieb-loss}. If $U=\RR^d$ we will sometimes omit the argument "$(\RR^d)$". For example, we will just write $L^p$, $W^{k,p}$ or $C^k_b$ instead of $L^p(\RR^d)$, $W^{k,p}(\RR^d)$ or $C^k_b(\RR^d)$. If $U\neq \RR^d$ we will not use this abbreviation. We will also use Banach space-valued Sobolev spaces as defined by L. C. Evans \cite[p.\,301-305]{evans}. For any Banach space $X$, $L^p(0,T;X)$ denotes the space of Banach space-valued $L^p$-functions $[0,T]\ni t\mapsto u(t) \in X$. Analogously, $W^{k,p}(0,T;X)$ denotes the Banach space-valued Sobolev space. The following properties are essential:
\begin{lem}
	\label{SOBBOCH}
	Let $T>0$, $d,m \in\NN$, $U\subset \RR^d$ be any open subset and $1\le p,q< \infty$ be arbitrary. Then the following holds:
	\begin{itemize}
		\item[\textnormal{(a)}] For any function $u\in L^p(0,T;W^{k,q}(\RR^d))$ there exists some sequence \linebreak$(u_j)\subset C^\infty(]0,T[\times \RR^d)$ such that
		\begin{gather*}
		\forall j\in\NN\; \exists r_j>0\; \forall t\in[0,T]: \supp u_j(t) \subset B_{r_j}(0)\\
		\text{and}\quad u_j \to u \quad \text{in}\; L^p(0,T;W^{k,q}(\RR^d)).
		\end{gather*}
		\item[\textnormal{(b)}] For any function $u\in W^{k,p}(0,T;L^q(\RR^d))$ there exists some sequence \linebreak$(u_j)\subset C^\infty(]0,T[\times \RR^d)$ such that
		\begin{gather*}
		\forall j\in\NN\; \exists r_j>0\; \forall t\in[0,T]: \supp u_j(t) \subset B_{r_j}(0)\\
		\text{and}\quad u_j \to u \quad \text{in}\; W^{k,p}(0,T;L^q(\RR^d)).
		\end{gather*}
		\item[\textnormal{(c)}] $L^p(]0,T[\times\RR^d)= L^p\big(0,T;L^p(\RR^d)\big)$.
		\item[\textnormal{(d)}] $W^{1,p}(]0,T[\times\RR^d) = W^{1,p}\big(0,T;L^p(\RR^d)\big) \cap L^p\big(0,T;W^{1,p}(\RR^d)\big)$.
	\end{itemize}
\end{lem}
\smallskip
The very technical proof is outsourced to the appendix.\\[-2mm]

In order to write down the three dimensional Vlasov-Poisson system concisely we will first define some operators and notations: For $d\in\NN$, $1\le p\le \infty$ and $r>0$ let $L^p_r(\RR^d)$ denote the set of functions $ \varphi\in L^p(\RR^d)$ having compact support $\supp \varphi\subset B_r(0)\subset\RR^d$. Then the operator
\begin{align}
\rho.\colon L^2_r(\RR^6) \to L^2(\RR^3),\;  \varphi \mapsto \rho_\varphi\quad\text{with}\quad \rho_\varphi(x):=\int  \varphi(x,v) \dv,\; x\in\RR^3
\end{align}
is linear and bounded. It also holds that $\rho_\varphi\in L^2_r(\RR^3)$ for any $ \varphi\in L^2_r(\RR^6)$. Let now $R>0$ be any arbitrary radius. From the Calderon-Zygmund inequality \linebreak\cite[p.\,230]{gilbarg-trudinger} we can conclude that
\begin{align}
\label{DEFPSI}
\psi.\colon L^2_r(\RR^6) \to H^2\big(\BR\big),\;  \varphi \mapsto \psi_\varphi \quad\text{with}\quad \psi_\varphi(x):= \int \frac{\rho_\varphi(y)}{|x-y|}\dy
\end{align}
is a linear and bounded operator. According to E. Lieb and M. Loss \cite[s.\,6.21]{lieb-loss}, the gradient of $\psi_\varphi$ is given by
\begin{align*}
\delx\psi_\varphi(x)= -\int \frac{x-y}{|x-y|^3}\;\rho_\varphi(y)\dy, \quad x\in \RR^3
\end{align*}
and then, because of \eqref{DEFPSI}, the operator
\begin{align}
\begin{aligned}
\delx\psi.\colon L^2_r(\RR^6) \to H^1\big(\BR;\RR^3\big),\;  \varphi \mapsto \delx\psi_\varphi \\
\end{aligned}
\end{align}
is also linear and bounded. Some more properties of the potential $\psi_\varphi$ and its field $\delx \psi_\varphi$ are given by the following lemma.

\begin{lem}
	\label{NPOT}
	Let $\varphi \in L^2_r(\RR^6)$ be arbitrary. 
	\begin{itemize}
		\item[\textnormal{(a)}] $\psi_\varphi\in H^2_{loc}(\RR^3)$ has a continuous representative and is the unique solution of the boundary value problem
		\begin{align*}
		-\laplace \psi_\varphi = 4\pi\rho_\varphi \;\;\text{a.e.\,on}\; \RR^3 ,\quad \underset{|x|\to\infty}{\lim} \psi_\varphi = 0.
		\end{align*}
		\item[\textnormal{(b)}] Let $1<p,q<\infty$ be any real numbers and suppose that additionally $\varphi\in L^p(\RR^6)$. Then ${\rho_\varphi \in L^p(\RR^3)}$ and there exists some constant $C>0$ depending only on $p$ and $r$ such that \begin{align*}
		\|\psi_\varphi\|_{L^p} &\le C\,\|\varphi\|_{L^q}, && \text{where}\quad \tfrac 1 q = \tfrac 2 3 +\tfrac 1 p\,, \text{ i.e.},\; q=\tfrac{3p}{2p+3}	\\
		\|\delx\psi_\varphi\|_{L^p} &\le C\,\|\varphi\|_{L^q}, && \text{where}\quad \tfrac 1 q = \tfrac 1 3 +\tfrac 1 p\,, \text{ i.e.},\; q=\tfrac{3p}{p+3}\\
		\|D_x^2 \psi_\varphi\|_{L^p} &\le C\,\|\varphi\|_{L^q } , && \text{where}\quad \tfrac 1 q = \tfrac 0 3 +\tfrac 1 p\,, \text{ i.e.},\; q=p .
		\end{align*}
		\item[\textnormal{(c)}] Suppose that additionally $\varphi\in L^\infty(\RR^6)$. Then $\rho_\varphi\in L^\infty(\RR^3)$ and there exists some constant $C>0$ depending only on $r$ such that $$\|\delx \psi\|_{L^\infty} \le C\,\|f\|_{L^\infty}.$$
		Moreover $\delx\psi_\varphi \in C^{0,\gamma}(\RR^3;\RR^3)$ for any $\gamma\in ]0,1[$ and there exists some constant $C>0$ depending only on $r$ such that 
		$$ |\delx\psi_\varphi(x) - \delx\psi_\varphi(y)| \le C\,\|\varphi\|_{L^\infty}\,|x-y|^\gamma, \quad x,y\in\RR^3. $$
	\end{itemize}
\end{lem}

\begin{proof}
	\textit{Item} (a): The Calderon-Zygmund inequality (\cite[p.\,230]{gilbarg-trudinger}) states that $\psi_\varphi$ is in $H^2_{loc}(\RR^3)$ and satisfies $	-\laplace \psi_\varphi = 4\pi\rho_\varphi$ almost everywhere on $\RR^3$. By Sobolev's inequality, $\psi_\varphi$ has a continuous representative and one can easily show that $\psi_\varphi(x) \to 0$ if $|x|\to \infty$, i.e., $\psi_\varphi$ satisfies the boundary value problem. Any other solution is then given by $\psi_\varphi + h$ where $\laplace h = 0$ almost everywhere and $h$ satisfies the boundary condition. Then, by Weyl's lemma, $h$ is a harmonic function and thus $h=0$ which means uniqueness. This proves (a). \newpage
	
	\textit{Item} (b): $\rho_\varphi\in L^p(\RR^6)$ with $\|\rho_\varphi\|_{L^p} \le C\, \|\varphi\|_{L^p}$ is a direct consequence of Jensen's inequality. The first two inequalities are already established by \linebreak E.~Stein~\cite[p.\,119]{stein}. Note that the Calderon-Zygmund inequality also yields $\|D^2 \psi_\varphi \|_{L^p(B_{2r}(0))} \le C\, \|\varphi\|_{L^p}$. Moreover, if $|x|\ge 2r$,
	\begin{align*}
	|\delxi \delxj \psi_\varphi(x)| \le \int\limits_{|y|<r} \frac{C}{|x-y|^3}\, |\rho_\varphi(y)| \dy \le C\, |x|^{-3}\, \|\rho_\varphi\|_{L^1} \le C\, |x|^{-3}\, \|\varphi\|_{L^p}\,.
	\end{align*}
	Thus $\|D^2 \psi_\varphi \|_{L^p(\RR^3\setminus B_{2r}(0))} \le C\, \|\varphi\|_{L^p}$ which completes the proof of (b).\\[-2mm]
	
	\textit{Item} (c): $\rho_\varphi \in L^\infty(\RR^6)$ is obvious. It holds that
	\begin{align*}
	|\delx\psi_\varphi(x)| &\le \|\rho_\varphi\|_{L^\infty} \int\limits_{|y|<r} |x-y|^{-2} \dy \le C\,\|\varphi\|_{L^\infty} \int\limits_{|y|<3r} |y|^{-2} \dy, && |x|\le 2r,\\
	|\delx\psi_\varphi(x)| &\le \int\limits_{|y|<r} \frac{|\rho_\varphi(y)|}{|x-y|^2} \dy \le C\, |x|^{-2}\, \|\rho_\varphi\|_{L^1} \le C\, |x|^{-2}\, \|\varphi\|_{L^\infty}, && |x|> 2r
	\end{align*}
	and then $\|\delx\psi_\varphi\|_{L^\infty} \le C\,\|\varphi\|_{L^\infty}$ immediately follows. The second assertion is established by E. Lieb and M. Loss \cite[s.\,10.2]{lieb-loss}.
\end{proof}

\smallskip

We will also use the notation $\rho_f$, $\psi_f$ and $\delx\psi_f$ for functions $f=f(t,x,v)$ with $t\ge 0$, $x,v\in\RR^3$. In this case we will write
\begin{align*}
\rho_f(t,x) = \rho_{f(t)}(x),\qquad
\psi_f(t,x) = \psi_{f(t)}(x),\qquad
\delx\psi_f(t,x) = \delx\psi_{f(t)}(x)
\end{align*}
for any $t$ and $x$. As already mentioned in the introduction we consider the following initial value problem:\\
\begin{equation}
\label{VPSU}
\begin{cases}
\partial_t f + v\cdot \partial_x f - \partial_x \psi_f \cdot \partial_v f + (v\times B) \cdot \partial_v f= 0 & \text{ on }[0,T]\times \RR^6\;,\\[0.25cm]
f|_{t=0} = \mathring f & \text{ on }\RR^6\,.
\end{cases}
\end{equation}
In the following let $T>0$ and $\mathring f\in C^2_c(\RR^6;\RR_0^+)$ be arbitrary but fixed. Let $B=B(t,x)$ be a given external magnetic field and let $f=f(t,x,v)$ denote the distribution function that is supposed to be controlled. Its electric field $\delx\psi_f$\linebreak$=\delx\psi_f(t,x)$ is formally defined as described above. In the following we will show that the solution $f$ satisfies the required condition "${f(t)=f(t,\cdot,\cdot)\in L^2_r(\RR^6)}$" that ensures $\rho_f$, $\psi_f$ and $\delx\psi_f$ to be well defined. Of course this is possible only if the magnetic field $B$ is regular enough. The regularity of those fields will be specified in the following section.

\section{Admissible fields and the field-state operator}

\subsection{The set of admissible fields}

We will now introduce the set our magnetic fields will belong to: The set of admissible fields. For $T>0$ and $\beta>3$ let $\WW=\WW(\beta)$ denote the reflexive Banach space $L^2\big(0,T;W^{2,\beta}(\RR^3;\RR^3)\big)$, let $\HH$ denote the Hilbert space $L^2\big(0,T;H^1(\RR^3;\RR^3)\big)$ and let $\|\cdot\|_\WW$ and $\|\cdot\|_\HH$ denote their standard norms. Then $\VV:=\WW \cap \HH$ with $\|\cdot\|_\VV := \|\cdot\|_\WW + \|\cdot\|_\HH$ is also a Banach space.

%
%

\begin{defn}
	\label{DAC}
	Let $K>0$ and $3<\beta<\infty$ be arbitrary fixed constants. Then,
	\begin{align*}
	\BB := \lbr B\in \VV \Bigvert \|B\|_\VV \le K \rbr
	\end{align*}
	is called the \textbf{set of admissible fields}.
\end{defn}

\begin{com}
	In the approach of Section 3 and 4 it would be sufficient to consider fields $B\in\WW$ with $\|B\|_\WW \le K$. However, in the model that is discussed in Section 5 the regularity $B\in\VV$ will be necessary. Therefore we will use this condition right from the beginning.
\end{com}

The most important properties of the set of admissible fields are listed in the following lemma.

\vspace{-2mm}

%
%

\begin{lem}
	\label{LAC}
	The set of admissible fields $\BB$ has the following properties:
	\begin{itemize}
		\item[\textnormal{(a)}]
		$\BB$ is a bounded, convex and closed subset of $\VV$.
		\item[\textnormal{(b)}]
		The space $W^{j,\beta}(\RR^3;\RR^3)$ is continuously embedded in $C^{j-1,\gamma}(\RR^3;\RR^3)$ for $j\in\NN$ and $\gamma=\gamma(\beta)=1-\tfrac{3}{\beta}$. Thus there exist constants $k_0,k_1>0$ depending only on $\beta$ such that for all $B\in\BB$,
		$$ \|B(t)\|_{C^{0,\gamma}} \le k_0\; \|B(t)\|_{W^{1,\beta}}, \quad \|B(t)\|_{C^{1,\gamma}} \le k_1\; \|B(t)\|_{W^{2,\beta}}$$
		for almost all $t\in[0,T]$. Moreover for any $r>0$ there exist constants $k_2,k_3>0$ depending only on $\beta$ and $r$ such that for all $B\in\BB$,
		\begin{align*}
		\|B(t)\|_{C^{0,\gamma}(B_r(0))} &\le k_2\; \|B(t)\|_{W^{1,\beta}(B_r(0))}, \\ \|B(t)\|_{C^{1,\gamma}(B_r(0))} &\le k_3\; \|B(t)\|_{W^{2,\beta}(B_r(0))}
		\end{align*}
		for almost all $t\in[0,T]$. 	
		\item[\textnormal{(c)}] The space $\WW$ is continuously embedded in $L^2(0,T;C^{1,\gamma})$. Thus for all \linebreak $B\in\BB$ it holds that $B\in L^2(0,T;C^{1,\gamma})$ with $ \|B\|_{L^2(0,T;C^{1,\gamma})} \le k_1 K \;. $
		\item[\textnormal{(d)}]
		Let $\MM$ denote the set
		\begin{align*}
		\left\{ B \in C^\infty([0,T]\times\RR^3;\RR^3)) \left|
		\begin{aligned}
		&\|B\|_\WW \le 2 K \;\text{and}\; \exists m>0\\
		&  \forall \in[0,T] : \supp B(t) \subset B_m(0)\subset \RR^3 
		\end{aligned}
		\right. \right\}.
		\end{align*}
		Then for any $B\in \BB$, there exists a sequence $(B_k)_{k\in\NN}\subset\MM$ such that
		$$\|B-B_k\|_\WW\to 0,\quad k\to\infty\;.$$
		\item[\textnormal{(e)}]
		$\BB\subset \VV$ is weakly compact, i.e., any sequence in $\BB$ contains a subsequence converging weakly in $\VV$ to some limit in $\BB$.
	\end{itemize}
\end{lem}

\begin{proof} 
	(a) is obvious and (b) is a direct consequence of Sobolev's embedding theorem and the fact that for all $B\in\BB$, $B(t)\in W^{2,\beta}(\RR^3;\RR^3)$ for almost all $t\in[0,T]$. Then the $k_1$-inequality of (b) immediately implies (c). (d) follows instantly from Lemma \ref{SOBBOCH}\,(a). Without loss of generality, we can assume that $\|B\|_\WW\le 2K$. As $\BB$ is a bounded subset of $\WW$ the Banach-Alaoglu theorem implies that any sequence ${(B_k)\subset\BB}$ contains a subsequence $(B_k^*)$ converging weakly to some limit $B\in \WW$. Now, $(B_k^*)$ is a bounded sequence in $\HH$ and thus it has a subsequence $(B_k^{**})$ that converges weakly to some limit in $\HH$. Because of uniqueness, this limit must be $B$. Hence $B^{**}_k\wto B$ in $\VV$ and since the norm $\|\cdot\|_\VV$ is weakly lower semicountinuous it follows that $\|B\|_\VV \le K$. That is (e).
\end{proof}

\subsection{The characteristic flow of the Vlasov equation}

Since the Vlasov equation is a first-order partial differential equation, it suggests itself to consider the characteristic system.
On that point, we will consider a general version of the Vlasov equation,
\begin{equation}
\label{VEF}
\partial_t f + v\cdot \partial_x f + F \cdot \partial_v f + v \times G \cdot \partial_v f= 0,
\end{equation}
with given fields $F=F(t,x)$ and $G=G(t,x)$. Then the following holds:

%
%

\begin{lem}
	\label{CHS}
	Let $I\subset \mathbb R$ be an interval and let $F,G\in C(I\times\mathbb R^3;\mathbb R^3)$ be continuously differentiable with respect to $x$ and bounded on $J\times \mathbb R^3$ for every compact subinterval $J\subset I$. Then for every $t\in I$ and $z=(x,v)\in \mathbb R^6$ there exists a unique solution $I\ni s\mapsto (X,V)(s,t,x,v)$ of the characteristic system
	\begin{equation}
	\dot x = v,\quad
	\dot v = F(s,x) + v\times G(s,x)
	\end{equation}
	to the initial condition $(X,V)(t,t,x,v)=(x,v)$. The characteristic flow \linebreak$Z:=(X,V)$ has the following properties:
	\begin{itemize}
		\item[\textnormal{(a)}]
		$Z:I\times I\times \mathbb R^6 \to \mathbb R^6$ is continuously differentiable.
		\item[\textnormal{(b)}]
		For all $s,t\in I$ the mapping $Z(s,t,\cdot):\mathbb R^6 \to \mathbb R^6$ is a $C^1$-diffeomorphism with inverse $Z(t,s,\cdot)$, and $Z(s,t,\cdot)$ is measure preserving, i.e.,
		\begin{equation*}
		\det \frac{\partial Z}{\partial z} (s,t,z) = 1, \quad s,t\in I,\; z\in\RR^6\;.
		\end{equation*}
	\end{itemize}
\end{lem}

The relation between the characteristic flow and the solution $f$ of the Vlasov equation \eqref{VEF} is described by the following lemma.

%
%

\begin{lem}
	\label{FL2}
	Under the assumptions of Lemma \ref{CHS} the following holds:
	\begin{itemize}
		\item[\textnormal{(a)}]
		A function $f\in C^1(I\times\mathbb R^6)$ satifies the Vlasov equation \eqref{VEF} iff it is constant along every solution of the characteristic system \eqref{CHS}.
		\item[\textnormal{(b)}]
		Suppose that $0\in I$. For $\mathring f\in C^1(\mathbb R^6)$ the function $f(t,z):=\mathring f(Z(0,t,z))$, $t\in I,z\in \mathbb R^6$
		is the unique solution of \eqref{VEF} in the space $C^1(I\times \mathbb R^6)$ with $f(0)=\mathring f$. If $\mathring f$ is nonnegative then so is f. For all $t\in I$ and $1\le p\le \infty$, 
		$$\textnormal{supp}f(t) = Z(t,0,\textnormal{supp} \mathring f) \tand \|f(t)\|_p = \|\mathring f\|_p\,.$$
		
	\end{itemize}
\end{lem}

For $G=0$ the proofs of both lemmata are presented by G. Rein \cite[p.\,394]{rein}. The proofs for a general field $G\in C(I\times\mathbb R^3;\mathbb R^3)$ proceed analogously.

\subsection{Classical solutions for smooth external fields}

As already mentioned in the introduction, the standard initial value problem $(\eqref{VP2}, \eqref{IC})$ posesses a unique global classical solution. This result holds true if the Vlasov equation is equipped with an external magnetic field $B$ in $C\big([0,T];C^1_b(\RR^3;\RR^3)\big)$ which will be established in the next theorem. Unfortunately the proof does not work if the field is merely an element of $\BB$. Since such fields are only $L^2$ in time, the same holds for the right-hand side of the characteristic system. This makes it impossible to determine a solution in the classical sense of ordinary differential equations. However, we can approximate any field $B\in \BB$ by a sequence $(B_k)_{k\in\NN}\subset\MM \subset C\big([0,T];C^1_b(\RR^3;\RR^3)\big)$ according to Lemma \ref{LAC}. This allows us to construct a certain kind of strong solution to the field $B$ as a limit of the classical solutions that are induced by the fields $B_k$.

%
%

\begin{thm}
	\label{GCS}
	Let $B \in C\big([0,T];C^1_b(\RR^3;\RR^3)\big)$ with $\|B\|_\WW\le 2K$ be arbitrary. Then the initial value problem \eqref{VPSU} possesses a unique classical solution \linebreak ${f\in C^1([0,T]\times\RR^6)}$. Moreover for all $t\in[0,T]$, ${f(t)=f(t,\cdot,\cdot)}$ is compactly supported in $\mathbb R^6$ in such a way that there exists some constant $R>0$ depending only on $T$, $\mathring f$, $K$ and $\beta$ such that for all $t\in[0,T]$,
	\begin{equation*}
	\textnormal{supp } f(t) \subset B_R^6(0) = \left\{(x,v)\in\mathbb R^6 : |(x,v)| < R\right\}.
	\end{equation*}
	If $B \in C\big([0,T];C^2_b(\RR^3;\RR^3)\big)$, then additionally  $f\in C\big([0,T];C^2_b(\RR^6)\big)$.
\end{thm}

\smallskip

\begin{proof} \textit{Step 1 - Local existence and uniqueness:} For the standard Vlasov-Poisson system $(\eqref{VP2}, \eqref{IC})$ the existence and uniqueness of a local classical solution was firstly established by R. Kurth \cite{kurth}. As the field ${B\in C\big([0,T];C^1_b(\RR^3;\RR^3)\big)}$ is regular enough the existence and uniqueness of a local classical solution to our problem can be proved analogously. In this thesis we will only sketch the most important steps of that proof. The idea is to define a recursive sequence by
	\begin{align*}
	f_0(t,z) := \mathring f(z) \tand f_{k+1}(t,z) := \mathring f \big(Z_k(0,t,z) \big),\quad k\in \NN_0
	\end{align*}
	for any $t\ge 0$ and $z=(x,v)\in\RR^6$ where $Z_k$ denotes the solution of
	\begin{align*}
	\dot z = {\begin{pmatrix} \dot x \\ \dot v \end{pmatrix}} = \begin{pmatrix} v \\ -\delx\psi_{f_k}(s,x) + v\times B(s,x) \end{pmatrix} \twith Z_k(t,t,z) = z\;.
	\end{align*}
	By induction we obtain that for any $k\in\NN_0$, $Z_k$ is continuously differentiable with respect to all its variables and $f_k\in C^1([0,T]\times\RR^6)$. Moreover, according to Lemma \ref{FL2}, $f_{k+1} \in C^1([0,T]\times\RR^6)$ is the unique solution of the initial value problem
	\begin{align*}
	\delt f + v\cdot\delx f - \delx\psi_{f_k}\cdot\delv f +(v\times B)\cdot\delv f = 0, \quad f\big\vert_{t=0} = \mathring f\,.
	\end{align*}
	We intend to prove that the sequence $(f_k)$ converges to the solution of the initial value problem \eqref{VPSU} if $k$ tends to infinity. Analogously to Kurth's proof we can show that there exists $\delta>0$ and functions $Z$ and $f$ with $Z\in C([0,\delta_0]^2\times\RR^6)$, $f\in C([0,\delta_0]\times\RR^6)$ for any $\delta_0<\delta$ such that
	\begin{align*}
	Z(s,t,z) = \underset{k\to\infty}{\lim} Z_k(s,t,z) \tand f(t,z) = \mathring f\big(Z(0,t,z)\big) = \underset{k\to\infty}{\lim} f_k(t,z)
	\end{align*}
	uniformely in $s$, $t$ and $z$. For any arbitrary $\delta_0<\delta$ it turns out that $(\delx\psi_{f_k})$ and $(D_x^2\psi_{f_k})$ are Cauchy sequences in $C_b([0,\delta_0]\times\RR^3)$. This implies that $\delx\psi_f$ and $D_x^2\psi_f$ lie in $C_b([0,\delta_0]\times\RR^3)$ and consequently ${Z\in C^1([0,\delta_0]^2\times \RR^6)}$. As $\delta_0$ was arbitrary this yields ${f\in C^1([0,\delta[\times \RR^6)}$. Thus $f$ is a local solution of the initial value problem \eqref{VPSU} on the time interval $[0,\delta[$ according to Lemma \ref{FL2} as it is constant along any characteristic curve. \bpskip
	
	\textit{Step 2 - Higher regularity:} If $B \in C([0,T];C^2_b)$ it additionally follows by induction that for all $k\in \NN_0$ and $s,t\in[0,T]$, 
	\begin{gather*}
	\delx\psi_{f_k}\in C([0,T];C^2_b(\RR^3;\RR^3)), \quad Z_k(s,t,\cdot) \in C^2(\RR^6), \quad f_k \in C([0,T];C^2_b(\RR^6)).
	\end{gather*}
	Moreover, if $\delta$ is sufficiently small, one can also show that $(D_x^3 \psi_{f_k})$ is a Cauchy sequence in $C_b([0,\delta_0]\times\RR^3)$ for any $\delta_0<\delta$. Thus we can conclude that 
	\begin{gather*}
	\delx\psi_{f}\in C([0,T];C^2_b(\RR^3;\RR^3)), \quad Z(s,t,\cdot) \in C^2(\RR^6), \quad f \in C([0,T];C^2_b(\RR^6)).\pskip
	\end{gather*}
	
	\textit{Step 3 - Continuation onto the interval $[0,T]$:} Obviously Batt's continuation criterion (cf.\;J. Batt \cite{batt}) also holds true in our case. This means that we can show that the solution exists on $[0,T]$ by the following argumentation:
	We assume that $[0,T^*[$ with $T^*\le T$ is the right maximal time interval of the local solution and we will show that
	\begin{align}
	\begin{aligned}
	P(t) :&= \max \{ |v| : (x,v)\in \text{supp } f(s),0\le s\le t\}\\[2mm]
	& = \max \{ |V(s,0,x,v)| : (x,v)\in \text{supp } \mathring f,0\le s\le t\}
	\end{aligned}
	\end{align}
	is bounded on $[0,T^*[$. But then, according to Batt, the solution $f$ can be extended beyond $T^*$ which is a contradiction as $T^*$ was chosen to be maximal. Hence we can conclude that the solution exists on the whole time interval $[0,T]$. \pskip
	
	For the standard Vlasov-Poisson system (without an external field) such a bound on $P(t)$ is established in the Pfaffelmoser-Schaeffer proof \cite{pfaffelmoser,schaeffer}. We will proceed analogously and single out one particle in our distribution. Mathematically, this means to fix a characteristic $(X,V)(s)=(X,V)(s,0,x,v)$ with $(X,V)(0)=(x,v)\in\text{supp }\mathring f$. Now suppose that $0\le\delta\le t< T^*$. In the following, constants denoted by $C$ may depend only on $\mathring f$, $T$, $K$ and $\beta$. The aim in the Pfaffelmoser-Schaeffer proof is to bound the difference ${|V(t)-V(t-\delta)|}$ from above by an expression in the shape of $C\delta P(t)^\alpha$ where $\alpha < 1$ is essential. In our case an analogous approach would merely yield some bound that is ideally in the fashion of $ C\delta P(t)^\alpha + C\sqrt{\delta}P(t)$ because of the additional field term in the $\dot v$ equation of the characteristic system. However, we can use the fact that an external magnetic field does not accelerate or slow down the particles. Only the direction of velocity is influenced by $B$ but not its magnitude. This is reflected in the following computation: For $s\in [t-\delta,t]$,
	\begin{align*}
	|V(s)|^2 & = |V(t-\delta)|^2 + \int\limits_{t-\delta}^s \ddtau |V(\tau)|^2 \dtau \\[0.15cm]
	& \le P(t-\delta)^2 + 2\int\limits_{t-\delta}^s \big|\partial_x \psi_f(\tau,X(\tau))\big||V(\tau)| \dtau \;.
	\end{align*}
	The quadratic version of Gronwall's lemma (cf. Dragomir \cite[p.\,4]{dragomir}) and the definition of $\delx\psi_f$ then impliy that
	\begin{align*}
	|V(t)| 
	\le P(t-\delta) + \int\limits_{t-\delta}^t \iint \frac{f(s,y,w)}{|y-X(s)|^2}\ \mathrm dw\mathrm dy \ds \;.
	\end{align*}
	By the change of variables $y = X(s,t,x,v)$ and $w = V(s,t,x,v)$,
	\begin{align}
	\label{INT}
	|V(t)| \le P(t-\delta) + \int\limits_{t-\delta}^t \iint \frac{f(t,x,v)}{|X(s,t,x,v)-X(s)|^2}\ \mathrm dv\mathrm dx \ds
	\end{align}
	since $f$ is constant along the measure perserving characteristic flow. Then $\rho_f$ can be bounded by $\| \rho_f (t)\|_{\infty} \le C P(t)^3$ for all $ t\in[0,T^*[$. Moreover
	\begin{align*}
	\| \partial_x \psi_f(t) \|_{\infty} \le C\ \| \rho_f (t)\|_{5/3}^{5/9}\
	\|\rho_f (t)\|_{\infty}^{4/9},\quad t\in[0,T^*[
	\end{align*}
	according to G. Rein \cite[pp.\,388-390]{rein}. In the Pfaffelmoser-Schaeffer proof, one further essential result is that $\|\rho_f(t)\|_{5/3}$ is bounded uniformely in $t$ on the time interval $[0,T^*[$ which is a consequence of energy conservation. Fortunately, the energy is still conserved in our Vlasov equation that is equipped with the magnetic field $B$. Mathematically, this means that
	\begin{align*}
	{\mathcal E}(t) &= \frac 1 2 \iint |v|^2 f(t)\ \mathrm dv \mathrm dx 
	+ \frac{1}{8\pi} \int |\partial_x \psi_f(t)|^2 \ \mathrm dx 
	\end{align*}
	does not depend on $t$. By an interpolation argument we obtain that
	\begin{align*}
	\| \rho_f(t) \|_{5/3} = \left( \int \rho_f(t)^{5/3} \mathrm dx \right)^{3/5}
	\le C\left( \iint |v|^2 f(t)\ \mathrm dv \mathrm dx \right)^{3/5}
	\le C\,\mathcal E(0)^{3/5} \le C
	\end{align*}
	for all $t\in[0,T^*[$ as presented by Rein \cite[p.\,416]{rein}. Consequently,
	\begin{align}
	\| \partial_x \psi_f(t) \|_{\infty}\le C^* P(t)^{4/3},\quad t\in[0,T^*[
	\end{align}
	for some positive constant $C^*$ that depends at most on $\mathring f$. Now, for any parameter $0<p\le P(t)$, we define
	$${\delta = \delta (t):= \min \left\{1,\frac t 2, \frac{p^2}{16(C^{*}+2k_1K)^2 P(t)^{8/3}} \right\}}\,.$$
	Without loss of generality we may assume that $P(t)\ge 1$ for all $t\in[0,T^*[$ (otherwise we replace $P(t)$ by $P(t)+1$). Thus
	\begin{align*}
	&\left|V(s,t,x,v)-v\right|  \;=\; \int\limits_s^t |\dot V(\tau,t,x,v)| \ \mathrm d\tau 
	\;\le\; \int\limits_s^t \|\delx\psi_f(\tau)\|_\infty + P(t)\|B(\tau)\|_\infty \dtau \\[2mm]
	& \quad \;\le\; \delta C^* P(t)^{4/3} + \sqrt{\delta}\;2k_1KP(t) \;\le\; \sqrt{\delta}\;(C^*+2k_1K) P(t)^{4/3} \;\le\; \frac p 4
	\end{align*}
	for all $s\in[t-\delta,t]$ and $(x,v)\in \RR^6$. Now, for any $r>0$, the Pfaffelmoser-Schaeffer proof yields the following bound:
	\begin{align*}
	\iint\frac{f(t,x,v)}{|X(s,t,x,v)-X(s)|^2} \dv \mathrm dx &\le C\delta \left( p^{4/3} + \delta r \ln\left( \frac{4P(t)}{p} \right) + \frac 1 {r\delta} \right).
	\end{align*}
	For a detailed derivation of this inequality confer Rein \cite[pp.\,418-422]{rein}. Thus inequality \eqref{INT} implies that
	\begin{align*}
	|V(t)| &\le P(t-\delta) + C \delta \left( p^{4/3} + r\;\ln\left( \frac{4P(t)}{p} \right) + \frac{1}{r\delta} \right)\;.
	\end{align*}
	We will now choose $p = P(t)^{4/7}$ and $r = P(t)^{16/21}$ in order to enforce that the terms of the sum on the right-hand side of this estimate are of the same order in $P(t)$. Then\vspace{-1mm}
	$${\delta(t) = \min \left\{1,\frac t 2, \frac{1}{16(C^{*}+2k_1K)^2} P(t)^{-32/21} \right\}}\;.$$
	Moreover, suppose that
	\begin{align*}
	P(t) \ge 1 + \big[ 16( C^* + 2k_1K )^2 \big]^{-{21}/{32}},\quad t\in[0,T^*[
	\end{align*}
	(otherwise we replace $P(t)$ by $P(t)+1 + [16( C^* + 2k_1K )^2 ]^{-{21}/{32}}\;$). This yields
	$${\delta(t) = \min \left\{\frac t 2, \frac{1}{16(C^{*}+2k_1K)^2} P(t)^{-32/21} \right\}}\;.$$
	\newpage\noindent
	If for all $t\in[0,T^*[$,\vspace{-1mm}
	\begin{align*}
	\frac{1}{16(C^{*}+2k_1K)^2} P(t)^{-32/21} > \frac t 2
	\end{align*}
	then immediately $P(t)\le C'$ on $[0,T^*[$ for some constant $C'>0$ depending only on $\mathring f$, $K$ and $\beta$. Else there exists
	\begin{align*}
	T':=\inf\left\{ t\in [0,T^*[\,\cap\, [0,T] \Bigvert \frac{1}{16(C^{*}+2k_1K)^2} P(t)^{-32/21} \le \frac t 2 \right\}\;.
	\end{align*}
	Since $P(t)$ is monotonically increasing,
	\begin{align*}
	\delta(t) = \frac{1}{16(C^{*}+2k_1K)^2} P(t)^{-32/21},\quad t\ge T'\;.
	\end{align*}
	Hence $\delta$ is decreasing on $[T',T^*[$. For $t\in[T',T^*[$,
	\begin{align*}
	& |V(t)| \le P(t-\delta(t)) + C \delta(t) \left( p^{4/3} + r\;\ln\left( \frac{4P(t)}{p} \right) + \frac{1}{r} P(t)^{32/21} \right)\\
	& \;\le P(t-\delta(t)) + C\delta(t)\;P(t)^{16/21}\ln(P(t)) \;\le\; P(t-\delta(t)) + C\delta(t)\; P(t)^{17/21}.
	\end{align*}
	Let now $t_0\in ]T',T^*[$ be arbitrary. We define $t_{i+1}:=t_i - {\delta(t_i)}$ as long as $t_i\ge T'$. Since $ t_i - t_{i+1} = {\delta (t_i)} \ge {\delta (t_0)} $, there exists $k\in \mathbb N$ such that
	\begin{equation*}
	t_{k+1}\le T'< t_k < \cdot\cdot\cdot <t_0.
	\end{equation*}
	Without loss of generality, $t_{k+1} = T'$ (otherwise we shrink $\delta(t_k)$ appropriately). Then for ${i\in\{1,...,k\}}$ and $t\in [t_{i+1},t_i]$,
	\begin{align*}
	|V(t)| \le P(t_{i+1}) + C\delta(t_{i+1}) P(t_0)^{17/21}\;.
	\end{align*}
	Additionally, $|V(t)|\le P(t_{i+1})$ if $t<t_{i+1}$ for any ${i\in\{1,...,k\}}$. Thus
	\begin{align*}
	P(t_i) - P(t_{i+1}) \le  C\delta(t_{i+1}) P(t_0)^{19/21}\;.
	\end{align*}
	Consequently\vspace{-2mm}
	\begin{align*}
	P(t_0) -\hspace{-1pt} P(T') &= \sum_{i=0}^k \big( P(t_i) -\hspace{-1pt} P(t_{i+1}) \big) \\
	&\le CP(t_0)^{19/21} \sum_{i=0}^k \delta(t_{i+1}) \;\le\; C t_0 P(t_0)^{19/21}.
	\end{align*}
	Since $t_0\in [T',T^*[$ was arbitrary, this means that for all $t\in [T',T^*[$,
	\begin {align*}
	P(t) \le P(T') + CtP(t)^{19/21} \le \big( P(T')^{2/21} + Ct \big) P(t)^{19/21} \le C''( 1 + t ) P(t)^{19/21}
\end{align*}
where $C''$ depends only on $T'$, $\mathring f$, $K$ and $\beta$. This finally yields
\begin{align*}
P(t) \le \max\{C',C''\} (1+t)^{21/2},\quad t\in [0,T^*[\;.
\end{align*}
Now according to Batt's continuation criterion the solution can be extended beyond $T^*$ which is a contradiction since $T^*$ was chosen as large as possible. This implies that the solution exists on the whole time interval $[0,T]$. Then 
\begin{align*}
T'=\inf\left\{ t\in [0,T] \Bigvert \frac{1}{16(C^{*}+2k_1K)^2} P(t)^{-32/21} \le \frac t 2 \right\}\;.
\end{align*}
depends only on $\mathring f$, $T$, $K$ and $\beta$ and thus
\begin {align*}
P(t) \le C(1+T)^{21/2} =:C_P,\quad t\in[0,T]
\end{align*}
where $C_P>0$ depends only on $\mathring f$, $T$, $K$ and $\beta$. We will now consider
\begin{align*}
Q(t) :&= \max\{|x|:(x,v)\in \supp f(s), 0\le s\le t\} 
,\quad t\in [0,T].
\end{align*}
Obviously $Q(t) \le Q(0) + TC_P =: C_Q$ for all $t\in[0,T]$ where $C_Q>0$ depends only on $T$ and $\mathring f$. Finally we define
\begin{align*}
S(t) := \max\{|(x,v)|:(x,v)\in \text{supp } f(s), 0\le s\le t\},\quad t\in [0,T]
\end{align*}
and obtain $S(t) \le P(t) + Q(t) < C_P + C_Q+1 =:R$ for all $t\in [0,T]$ which means that $\text{supp }f(t) \subset B_R^6(0)$ for all $t\in[0,T]$. 
\end{proof}

Temporarily we will write $f_B$ to denote the classical solution that is induced by the field~$B$. In order to prove that any field $B\in\BB$ still induces a strong solution of the initial value problem the following two lemmata are essential. For fields $B\in\MM$ Lemma \ref{LIP} asserts that $f_B$ depends Lipschitz continuously on $B$ while its derivatives $\delz f_B$ and $\delt f_B$ are Hölder continuous with respect to $B$. In the course of the construction of a strong solution to some field $B\in\BB$ we will approximate $B$ by a sequence $(B_k)\subset\MM$ and then Lemma \ref{LIP} will ensure that $(f_{B_k})$, $(\delt f_{B_k})$ and $(\delz f_{B_k})$ are Cauchy sequences in some sense. To prove Lemma \ref{LIP} we will need some uniform bounds that are established in Lemma \ref{ES1}.

%
%

\begin{lem}
\label{ES1}
Let $B\in \MM$ be any smooth field. For $t,s\in[0,T]$, $z=(x,v)\in\RR^6$, let $Z_B=Z_B(s,t,z)$ $=(X_B,V_B)(s,t,x,v)$ be the solution of the characteristic system with $Z_B(t,t,z)=z$. Furthermore let $f_B$ be the classical solution of the initial value problem \eqref{VPSU} to the field $B$. Then, there exist constants $R_Z\ge R$, $c_1,c_2,c_3,c_4>0$ depending only on $\mathring f$, $T$, $\K$, and $\beta$ such that for all $t,s\in [0,T]$,\\[-0.3cm]
\begin{gather*}
\|Z_B(s,t,\cdot)\|_{L^\infty(B^6_R(0))} \le R_Z\,,\quad
\|D_z Z_B(s,t,\cdot)\|_{L^\infty(B^6_R(0))}\le c_1\,,\\
\|\delz f_B(t)\|_\infty \le c_2\,,\quad
\|D_x^2 \psi_{f_B}(t)\|_\infty \le c_3\,,\quad
\|\delt f_B\|_{L^2(0,T;C_b)}\le c_4 \,.
\end{gather*}
$Z_B$ and $f_B$ are twice continuously differentiable with respect to $z$ and there exist constants $c_5,c_6>0$ depending only on $\mathring f$, $T$, $\K$, and $\beta$ such that for all $s\in[0,T]$,
\begin{gather*}
\big\|t\mapsto Z_B(s,t,\cdot)\big\|_{L^\infty(0,T;W^{2,\beta}(\BR))}\le c_5, \quad \|D_z^2 f_B\|_{L^\infty(0,T;L^\beta)} \le c_6\;.
\end{gather*}
\end{lem}

\begin{proof}
\label{PES1}
Let $s,t\in [0,T]$ and ${z\in \BR}$ be arbitrary (without loss of generality $s\le t$) and let $i,j\in {1,...,6}$ be arbitrary indices. Let $B\in\MM$ be an arbitrary field and let $Z_B\colon [0,T] \times [0,T] \times \RR^6 \to \RR^6$ denote the induced solution of the characteristic system satisfying the initial condition $Z_B(t,t,z) = z$.
For brevity, we will use the notation $Z_B(s)=Z_B(s,t,z)$. The letter $C$ will denote a positive generic constant depending only on $\mathring f$, $K$, $T$ and $\beta$. It holds that
\begin{align*}
|Z_B(s)|^2 &\le |z|^2 + \hspace{-0.5mm}\int\limits_s^t \ddtau |Z_B(\tau)|^2\ \mathrm d\tau \\
& \le R^2 + \int\limits_s^T |Z_B(\tau)|^2\ \mathrm d\tau + \int\limits_s^T |Z_B(\tau)|\|\partial_x \psi_{f_B}(\tau) \|_\infty\dtau.
\end{align*}
Hence applying first the standard version and then the quadratic version of Gronwall's lemma provides that
\begin{align}
\|Z_B(s)\|_{L^\infty(\BR)} \le C + C \int\limits_s^T \|\partial_x \psi_{f_B}(\tau) \|_\infty\dtau \le C  =: R_Z.
\end{align}
The partial derivatives $\delzi Z_B$, $i=1,...,6$ can be bounded by
\begin{align*}
|\partial_{z_i} Z_B(s)| \le 1 + \int\limits_s^t C \big( 1 + \|D_x^2 \psi_{f_B}(\tau)\|_\infty + \|B(\tau)\|_{W^{1,\infty}} \big) |\partial_{z_i} Z_B(\tau)|\;\mathrm d\tau
\end{align*}
and consequently, by Gronwall's lemma,
\begin{align}
\label{ADZ}
\|D_z Z_B(s)\|_{L^\infty(\BR)} 
\le C \exp\left( C \int\limits_s^t \|D_x^2 \psi_{f_B}(\tau)\|_\infty \;\mathrm d\tau \right)
\end{align}
According to G. Rein in \cite[p.\,389]{rein},
\begin{align}
\label{AD2PSI}
\|D_x^2 \psi_{f_B}(t)\|_\infty &\le C\left[ (1+\|\rho_{f_B}(t)\|_\infty)(1+ \ln_+\|\partial_x \rho_{f_B}(t)\|_\infty) + \|\rho_{f_B}(t)\|_{L^1} \right] \nonumber \\
&\le C+C\ln_+\|\partial_x \rho_{f_B}(t)\|_\infty
\end{align}
as $\|\rho_{f_B}(t)\|_\infty \le \tfrac 4 3 R^3 \pi \|\mathring f\|_\infty$ and $\|\rho_{f_B}(t)\|_{L^1} = \|f_B(t)\|_{L^1} = \|\mathring f\|_{L^1}$. Moreover,
\begin{align*}
&|\partial_x \rho_{f_B}(t,x)|  \le \int\limits_{|v|\le R} |\partial_x f(t,x,v)|\;\mathrm dv
\le \frac{4\pi}{3}R^3 \; \|\partial_z \mathring f\|_\infty \|\partial_z Z_B(0)\|_{L^\infty(\BR)} \\
&\quad \le \frac{4\pi}{3}R^3\; \|\partial_z \mathring f\|_\infty\; C \exp\left( C \int\limits_0^t \|\partial_x^2 \psi_{f_B}(\tau)\|_\infty\;\mathrm d\tau \right)
\end{align*}
and now, by (\ref{AD2PSI}) and Gronwall's lemma, $\|D_x^2 \psi_{f_B}(t)\|_\infty \le C =: c_3$ for all ${t\in [0,T]}$. Thus, due to \eqref{ADZ}, $\|D_z Z_B(s) \|_{L^\infty(\BR)} \le C =: c_1$ for all ${t\in [0,T]}$. 
This directly yields
\begin{align}
\label{DELZF}
\| \partial_z f_B(t) \|_\infty \le \| \partial_z \mathring f \|_\infty  \|D_z Z_B(0) \|_{L^\infty(\BR)} \le C =: c_2\;.
\end{align}
and we can finally conclude that $\| \delt f_B \|_{L^2(0,T;C_b)} \le C =: c_4$ by expressing $\delt f$ by the Vlasov equation. In Step 2 of the proof of Theorem \ref{GCS} we have already showed that for all $s,t\in [0,T]$,
\begin{gather*}
\delx\psi_{f_B}\in C([0,T];C^2_b(\RR^3;\RR^3)), \quad Z_B(s,t,\cdot) \in C^2(\RR^6), \quad f_B \in C([0,T];C^2_b(\RR^6)).
\end{gather*}
Let $i,j,k\in \{1,...,6\}$ be arbitrary. Recall that, according to Lemma \ref{NPOT},
\begin{align}
\label{ESTPSI}
\begin{aligned}
\| \delxi \delxj \partial_{x_k} \psi_{f_B} (t) \|_{L^\beta} &= \| \delxi \delxj \psi_{\partial_{x_k} f_B} (t) \|_{L^\beta} \le C\, \| \partial_{x_k} f_B (t) \|_{L^\beta}\\
& \le C\,\| \partial_{x_k} f_B (t) \|_{L^\infty} \le C.
\end{aligned}
\end{align}
Now, for all $s,t\in[0,T]$ (without loss of generality $s\le t$),
\begin{align*}
&\int\limits_\BR |\delzi\delzj Z_B(s,t,z)|^\beta \dz = \int\limits_s^t \ddtau \int\limits_\BR |\delzi\delzj Z_B(\tau,t,z)|^\beta \dz \dtau \\
&\quad = \beta \int\limits_s^t \int\limits_\BR |\delzi\delzj Z(\tau,t,z)|^{\beta-1}\, |\delzi\delzj \dot Z_B(\tau,t,z)| \dz \dtau.
\end{align*}
Note that for all $s,t\in[0,T]$ and $z\in\BR$,
\begin{align*}
|\delzi\delzj \dot Z_B(\tau,t,z)| 
&\le C\, |\delzi\delzj Z_B(\tau)| \Big( 1 + \|D_x^2 \psi_{f_B}(\tau)\|_{L^\infty} + \|D_x B(\tau)\|_{L^{\infty}} \Big) \\
&\quad + C\, \Big( |D_x^3 \psi_{f_B}(\tau,X_B(\tau))| + |D_x^2 B(\tau,X_B(\tau))| \Big)\,.
\end{align*}
Thus, applying Hölder's inequality with exponents $p=\frac \beta{\beta-1}$ and $q=\beta$ gives
\begin{align*}
&\int\limits_\BR |\delzi\delzj Z_B(s,t,z)|^\beta \dz \\
&\;\;\le C\int\limits_s^t \left(\; \int\limits_\BR |\delzi\delzj Z_B(\tau)|^{\beta} \dz \right) \Big( 1 + \|D_x^2 \psi_{f_B}(\tau)\|_{L^\infty} + \|B(\tau)\|_{W^{2,\beta}} \Big) \dtau \\
&\quad + C\int\limits_s^t \left(\;\int\limits_\BR |\delzi\delzj Z_B(\tau)|^{\beta} \dz \right)^{\frac{\beta-1}{\beta}} \Big( \|D_x^3 \psi_{f_B}(\tau)\|_{L^{\beta}} + \|B(\tau)\|_{W^{2,\beta}} \Big) \dtau
\end{align*}
Now we can use \eqref{ESTPSI} and a nonlinear generalization of Gronwall's lemma (cf. \cite[p.\,11]{dragomir}) with exponent $\frac{\beta-1}{\beta}\in ]0,1[ $ to obtain that $\|\delzi\delzj Z_k(s,t,\cdot)\|_{L^\beta} \le C$. This finally implies that
\begin{align*}
\big\|t\mapsto Z_B(s,t,\cdot)\big\|_{L^\infty(0,T;W^{2,\beta}(\BR))} \le C=:c_5, \quad s\in[0,T]
\end{align*}
Finally, by chain rule,
\begin{align*}
&\|\delzi \delzj f\|_{L^\infty(0,T;L^\beta)} \le \|\mathring f\|_{C^2_b}\; \big\| t\mapsto Z(0,t,\cdot) \big\|_{L^\infty(0,T;W^{2,\beta}(\BR))} 
\le C=:c_6\;.
\end{align*}
The proof is complete.
\end{proof}

%
%

\begin{lem}
\label{LIP}
Let $B,H\in \MM$ and let $f_B,f_H$ be the induced classical solutions. Moreover, let $Z_B$ denote the solution of the characteristic system to the field $B$ satisfying $Z_B(t,t,z)=z$ and let $Z_H$ be defined analogously. Then, there exist constants $\ell_1,\ell_2,L_1,L_2,L_3>0$ depending only on $\mathring f$, $T$, $\K$ and $\beta$ such that
\begin{align*}
\|Z_B-Z_H\|_{C([0,T];C_b(\BR))} &\;\le\; \ell_1 \|B-H\|_\WW\;,\\[1mm]
\|\partial_z Z_B-\partial_z Z_H\|_{C([0,T];C_b(\BR))} &\;\le\; \ell_2 \|B-H\|_\WW^{\gamma}\;,\\[1mm]
\|f_B-f_H\|_{C([0,T];C_b)} &\;\le\; L_1 \|B-H\|_\WW\;,\\[1mm]
\|\partial_z f_B-\partial_z f_H\|_{C([0,T];C_b)} &\;\le\; L_2 \|B-H\|_\WW^\gamma\;,\\[1mm]
\|\partial_t f_B-\partial_t f_H\|_{L^2(0,T;C_b)} &\;\le\; L_3 \|B-H\|_\WW^\gamma\;
\end{align*}
where $\gamma=\gamma(\beta)$ is the H\"older exponent from Lemma \ref{LAC}.
\end{lem}

\begin{proof}
Let $B,H\in \MM$, $s,t\in[0,T]$ and $z\in\BR$ be arbitrary. Without loss of generality $s\le t$. Moreover, let $C>0$ denote a generic constant depending only on $\mathring f$, $T$ and $K$ and let $Z_B$ and $Z_H$ denote the solutions of the characteristic system to the fields $B$ and $H$ satisfying $Z_B(t,t,z) =z$ and $Z_H(t,t,z) =z$. Using Lemma \ref{ES1} and Lemma \ref{LAC} we obtain\vspace{-2mm}
\begin{align*}
&|Z_B(s) - Z_H(s)| \le \int\limits_s^t |\dot Z_B(\tau) - \dot Z_H(\tau)| \dtau\\
&\le \int\limits_s^t |V_B(\tau) - V_H(\tau)| + |\partial_x \psi_{f_B}(\tau,X_B(\tau))-\partial_x \psi_{f_H}(\tau,X_H(\tau))| \\[-3mm]
&\qquad\qquad\quad +|V_B(\tau)\times B(\tau,X_B(\tau)) - V_H(\tau)\times H(\tau,X_H(\tau))| \dtau \\[3mm]
&\le
C \int\limits_s^t \big( 1 + \|D_x^2 \psi_{f_B}(\tau)\|_\infty + \|B(\tau)\|_{W^{1,\infty}} \big)\; |Z_B(\tau) - Z_H(\tau)| \\[-3mm]
&\qquad\qquad\quad +\|\partial_x \psi_{f_B-f_H}(\tau)\|_{L^\infty(B_{R_Z}(0))}  + \|B(\tau)-H(\tau)\|_{L^\infty(B_{R_Z}(0))} \dtau 
\end{align*}
\pskip
Thus by Lemma \ref{NPOT}\,(c), Lemma \ref{ES1} and Gronwall's lemma,
\begin{align*}
|Z_B(s) - Z_H(s)| &\le C \int\limits_s^t \|{f_B}(\tau)-{f_H}(\tau)\|_\infty \mathrm d\tau + C \|B-H\|_\WW
\end{align*}
and then by chain rule,
\begin{align*}
&\|f_B(t)-f_H(t)\|_\infty \le \|D\mathring f\|_\infty\; \|Z_B(0,t,\cdot) - Z_H(0,t,\cdot)\|_{L^\infty(\BR)}\\
&\quad \le C \int\limits_0^t \|{f_B}(\tau)-{f_H}(\tau)\|_\infty \dtau + C \|B-H\|_\WW
\end{align*}
which yields
\begin{align}
\label{fB-fH}
\|f_B-f_H\|_{C([0,T];C_b)}&\le L_1 \|B-H\|_\WW\,, \\
\label{ZB-ZH}
\|Z_B-Z_H\|_{C([0,T];C_b(\BR))} &\le \ell_1 \|B-H\|_\WW\,,
\end{align}
if $\ell_1$ and $L_1$ are chosen suitably. Hence, by Lemma \ref{NPOT}\,(d),
\begin{align}
\label{HOELDPSI}
&|\delxi\delxj \psi_{f_B}(\tau,X_B(\tau)) - \delxi\delxj   \psi_{f_B}(\tau,X_H(\tau))| \notag \\[1mm]
&\quad= |\delxi\psi_{\delxj f_B}(\tau,X_B(\tau)) - \delxi \psi_{\delxj f_B}(\tau,X_H(\tau))|  \notag  \\
&\quad\le C\;|X_B(\tau) - X_H(\tau)|^{\gamma} \le C\;\|B-H\|_{L^2 ( 0,T;W^{1,\beta}(B_{R_Z}(0)) ) }^{\gamma}
\end{align}
for every $\tau\in[0,T]$ and every $i,j\in\{1,...,6\}$. Let now $i\in\{1,...,6\}$ be arbitrary. By Lemma \ref{ES1},
\begin{align*}
&|\delzi Z_B(s)-\delzi Z_H(s)| \le  \int\limits_s^t |\delzi \dot Z_B(\tau) - \delzi \dot Z_H(\tau)| \dtau \\
&\le
C \int\limits_s^t  \big(1+ \|D_x^2 \psi_{f_B}(\tau)\|_\infty + \|B(\tau)\|_{W^{1,\infty}}\big)\; | \delzi Z_B(\tau) - \delzi Z_H(\tau) |   \\[-4mm]
&\qquad \qquad \quad +  | D_x^2 \psi_{f_B}(\tau,X_B) - D_x^2 \psi_{f_H}(\tau,X_H) |  + |B(\tau,X_B) - H(\tau,X_H) |  \\
&\qquad \qquad \quad +  |D_x B(\tau,X_B) - D_x H(\tau,X_H) |  + \|D_x B(\tau)\|_\infty\; |V_B - V_H|  \dtau
\end{align*}
Now, applying \eqref{ZB-ZH}, \eqref{HOELDPSI}, Lemma \ref{LAC}\,(b) and Lemma \ref{ES1} yields
\begin{align*}
&|\delzi Z_B(s)-\delzi Z_H(s)| \\
& \le  C\int\limits_s^t  \big(1 + \|B(\tau)\|_{W^{2,\beta}}\big)\; | \delzi Z_B(\tau) - \delzi Z_H(\tau) | \dtau \\
&\quad + C \int\limits_s^t \|\delz f_B(\tau) - \delz f_H(\tau)\|_\infty \dtau \;+\; C\;\|B-H\|_\WW^{\gamma}
\end{align*}
and by Gronwall's lemma,
\begin{align}
\label{DZB-DZH-0}
|\delzi Z_B(s) - \delzi Z_H(s)| \le C\, \|B-H \|_\WW^{\gamma} + C \int\limits_s^t \|\delz f_B(\tau) - \delz f_H(\tau)\|_\infty\,.
\end{align}
Finally, by \eqref{ZB-ZH} and \eqref{DZB-DZH-0},
\begin{align*}
&\|\delz f_B(t) - \delz f_H(t)\|_\infty = \|\delz f_B(t) - \delz f_H(t)\|_{L^\infty(B_R(0))} \\[0.3cm]
&\quad\le C \|\delz Z_B(0) - \delz Z_H(0)\|_{L^\infty(B_R(0))} + C \|Z_B(0) - Z_H(0)\|_{L^\infty(B_R(0))} \\
&\quad\le C\, \| B  - H \|_\WW^{\gamma} + C \int\limits_0^t \|\delz f_B(\tau) - \delz f_H(\tau)\|_\infty\;,\vspace{-3mm}
\end{align*}
and consequently
\begin{align}
\label{DfB-DfH}
\|\delz f_B - \delz f_H\|_{C([0,T];C_b)} &\le L_2 \| B  - H \|_\WW^{\gamma},\\
\label{DZB-DZH}
\|\delzi Z_B - \delzi Z_H\|_{C([0,T];C_b(\BR))} &\le \ell_2 \| B  - H \|_\WW^{\gamma}
\end{align}
if $\ell_2$ and $L_2$ are chosen appropriately. The third $L_3$-inequality follows directly from \eqref{fB-fH} and \eqref{DfB-DfH} by representing $\delt f_B$ and $\delt f_H$ by their corresponding Vlasov equation.
\end{proof}

\subsection{Strong solutions for admissible external fields}

Now we will show that any field $B\in\BB$ still induces a unique strong solution which can be constructed as the limit of solutions $f_{B_k}$ where $(B_k)\subset\MM$ with $B_k\to B$ in $\WW$. Such a strong solution is defined as follows:

%
%

\begin{defn}
\label{WKS}
Let $B\in\BB$ be any admissible field. We call $f$ a \textbf{strong solution} of the initial value problem \eqref{VPSU} to the field $B$, iff the following holds:
\begin{enumerate}
\item[\textnormal{(i)}]		For all $1\le p \le \infty$, $f\in W^{1,2}(0,T;L^p(\RR^6)) \cap L^2(0,T;W^{1,p}(\RR^6))$\\ $\subset C([0,T];L^p(\RR^6))$ with
$$\|f\|_{W^{1,2}(0,T;L^p)} + \|f\|_{L^2(0,T;W^{1,p})} \le C$$ for some constant $C>0$ depending only on $\mathring f$, $T$, $K$ and $\beta$.
\item[\textnormal{(ii)}]	$f$ satisfies the Vlasov equation
$$\delt  f + v\cdot \partial_x  f - \partial_x \psi_f \cdot \partial_v  f + (v\times B) \cdot \partial_v  f= 0$$
almost everywhere on $[0,T]\times\RR^6$.
\item[\textnormal{(iii)}]	$f$ satisfies the initial condition $f\big\vert_{t=0}=\mathring f$ almost everywhere on $\RR^6$,
\item[\textnormal{(iv)}]	For every $t\in[0,T]$, $\supp f(t) \subset \BR$ where $R$ is the constant from Theorem \ref{GCS}.
\end{enumerate}
\end{defn}

\smallskip

First of all one can easily establish that such a strong solution is unique.

%
%

\begin{prop}
\label{UNIQS}
Let $B\in\BB$ be any field and suppose that there exists a strong solution $f$ of the initial value problem \eqref{VPSU} to the field $B$. Then this solution is unique.
\end{prop}

\begin{proof} 
Suppose that there exists another strong solution $g$ to the field $B$. Then the difference $h:=f-g$ satisfies
\begin{align*}
\delt h + v\cdot\delx h -\delx\psi_h \cdot\delv f -\delx \psi_g \cdot \delv h +(v\times B)\cdot\delv h = 0\;.
\end{align*}
almost everywhere on $[0,T]\times\RR^6$. Thus by integration by parts,
\begin{align*}
&\ddt \|h(t)\|_{L^2}^2 = 2\int \delt h \; h \dz = 2 \int \delx \psi_h\cdot \delv f\; h\dz \\
&\quad \le 2 \; \|\delv f(t)\|_{\infty}\; \|\delx\psi_h(t)\|_{L^2}\; \|h(t)\|_{L^2} \le C(R)\; \|\delv f(t)\|_\infty\;  \|h(t)\|_{L^2}^2\;.
\end{align*}
As $t\mapsto \|h(t)\|_{L^2}^2$ is continuous with $\|h(0)\|_{L^2}^2 =0$, Gronwall's lemma yields $\|h(t)\|_{L^2}^2=0$ for all $t\in[0,T]$. Hence for all $t\in[0,T]$, $f(t) = g(t)$ almost everywhere on $\RR^6$ which means uniqueness.
\end{proof}

Now we will show that any admissible field $B\in\BB$ actually induces a unique strong solution. Note that this solution is even more regular than it was demanded in the definition. However the weaker requirements of the definition will be essential in the later approach (see Proposition \ref{COMP}) and it will also be important that uniqueness was established under those weaker conditions.

%
%

\begin{thm}
\label{GWS}
Let $B\in\BB$. Then there exists a unique strong solution $f$ of the initial value problem \eqref{VPSU} to the field $B$. Moreover this solution satisfies the following properties which are even stronger than the conditions that are demanded in Definition \ref{WKS}:
\begin{enumerate}
\item[\textnormal{(a)}]		$f\in W^{1,2}(0,T;C_b(\RR^6))\cap C([0,T];C^1_b(\RR^6))\cap L^\infty(0,T;W^{2,\beta}(\RR^6))$ with
\begin{gather*}
\|f(t)\|_p = \|\mathring f\|_p\;,\quad t\in [0,T]\;,\quad 1\le p\le \infty\\
\text{and}\quad
\|f\|_{W^{1,2}(0,T;C_b)} + \|f\|_{C([0,T];C^1_b)} + \|f\|_{L^\infty(0,T;W^{2,\beta})} \le C
\end{gather*}
for some constant $C>0$ depending only on $\mathring f$, $T$, $K$ and $\beta$.
\item[\textnormal{(b)}]	$f$ satisfies the initial condition $f\big\vert_{t=0}=\mathring f$ everywhere on $\RR^6$,
\end{enumerate}
\end{thm}

\begin{proof} Let $B\in\BB$ arbitrary. According to Lemma \ref{LAC}, we can choose some sequence ${(B_k)_{k\in\NN}\subset \MM}$ with $B_k\to B$ for $k\to\infty$ in $\WW$.
Now Lemma~\ref{LIP} and Lemma~\ref{ES1} provide that for all $t\in[0,T]$ and $j,k\in\NN$,
\begin{align*}
\|f_{B_k}-f_{B_j}\|_{C([0,T];C_b)} &\le L_1 \|B_k-B_j\|_\WW\;, \\
\|\partial_z f_{B_k}-\partial_z f_{B_j}\|_{C([0,T];C_b)} &\le L_3 \|B_k-B_j\|_\WW^{\gamma}\;,\\
\|\partial_t f_{B_k}-\partial_t f_{B_j}\|_{L^2(0,T;C_b)} &\le L_4 \|B_k-B_j\|_\WW\;,\\
\|D_z^2 f_{B_k}\|_{L^\infty(0,T;L^\beta)} &\le c_6\;.
\end{align*}
where $\gamma=\gamma(\beta)$ is the constant from Lemma \ref{LAC}. Hence, $(f_{B_k})_{n\in\NN}$ is a Cauchy sequence in $C([0,T];C^1_b)\cap W^{1,2}(0,T;C_b)$. Due to completeness there exists a unique function ${f\in C([0,T];C^1_b)\cap W^{1,2}(0,T;C_b)}$ such that $f_{B_k}\to f$ in this space. Since $(f_{B_k})$ is also bounded in $L^\infty(0,T;W^{2,\beta})$ by some constant depending only on $\mathring f$, $T$, $K$ and $\beta$, the Banach-Alaoglu theorem states that there exists some function ${\bar f\in L^\infty(0,T;W^{2,\beta})}$ such that $f_{B_k}\overset{*}{\rightharpoonup} \bar f$ up to a subsequence. This means that for any ${\alpha\le 2}$, the sequence $(D_z^\alpha f_{B_k})$ converges to $D_z^\alpha \bar f$ with respect to the weak-*-topology on $[L^1(0,T;L^{\beta'})]^* \conf L^\infty(0,T;L^\beta)$ where $\nicefrac{1}{\beta} +\nicefrac{1}{\beta'} =1$. Because of uniqueness of the limit it holds that $D_z^\alpha f = D_z^\alpha \bar f$ and thus
$$ f = \bar f \in W^{1,2}(0,T;C_b)\cap C([0,T];C^1_b) \cap L^\infty(0,T;W^{2,\beta}).$$
To show that $f$ is a strong solution to the field $B$, we have to verify the conditions from Definition \ref{WKS}.  The strong convergence of $(f_{B_k})$ in $W^{1,2}(0,T;C_b)$ and $C([0,T];C^1_b)$ directly implies condition (ii). Moreover,\vspace{-1mm}
\begin{align*}
\big| f(0) - \mathring f \big| = \big| f(0) - f_{B_k}(0) \big| \le C\; \| f-f_{B_k} \|_{C([0,T];C_b)} \to 0,\qquad k\to \infty.
\end{align*}
Thus $f(0)=\mathring f$ everywhere on $\RR^6$ that is (b) which directly implies (iii). Due to uniform convergence and continuity of $f$, it is evident that $f(t)$ is also compactly supported in $B_R(0)$ for every $t\in [0,T]$. That is (iv). For $1\le q\le \infty$ arbitrary, $t\in [0,T]$ and $k\in\NN$, we have
\begin{align*}
&\Big| \|f(t)\|_q - \| \mathring f \|_q \Big| = \Big| \|f(t)\|_q - \| f_{B_k}(t) \|_q \Big|\le \|f(t) - f_{B_k}(t) \|_q \\
&\quad \le \big( \lambda(B_R(0)) \big)^{\frac 1 q} \|f(t) - f_{B_k}(t) \|_\infty  \le \big( 1+\lambda(B_R(0)) \big) \|f - f_{B_k} \|_{C([0,T];C_b)} \to 0,
\end{align*}
if $k\to\infty$ where $\frac 1 q \conf 0$ if $q=\infty$. This means that $\|f(t)\|_q = \| \mathring f \|_q$ for every $t\in[0,T]$ and ${1\le q\le\infty}$. Moreover we can choose some fixed $k\in\NN$ such that
$$\|f-f_{B_k}\|_{W^{1,2}(0,T;C_b)} + \|f-f_{B_k}\|_{C([0,T];C^1_b)} \le 1\;.$$
From Lemma \ref{ES1} we know that $\|f_{B_k}\|_{C([0,T];C^1_b)} \le C$, $\|f_{B_k}\|_{W^{1,2}(0,T;C_b)} \le C$ and $\|f_{B_k}\|_{L^\infty(0,T;W^{2,\beta})} \le C$. Thus
\begin{align*}
\|f\|_{W^{1,2}(0,T;C_b)} &\le \|f-f_{B_k}\|_{W^{1,2}(0,T;C_b)} + \|f_{B_k}\|_{W^{1,2}(0,T;C_b)} \le C, \\
\|f\|_{C([0,T];C^1_b)} &\le \|f-f_{B_k}\|_{C([0,T];C^1_b)}+ \|f_{B_k}\|_{C([0,T];C^1_b)}\le C.
\end{align*}
Moreover by the weak-* lower semicontinuity of the norm,
\begin{align*}
\|f\|_{L^\infty(0,T;W^{2,\beta})} &\le \underset{k\to\infty}{\lim\inf} \|f_{B_k}\|_{L^\infty(0,T;W^{2,\beta})}
\le C. 
\end{align*}
This proves (a) which includes condition (i). Finally, uniqueness follows directly from Proposition \ref{UNIQS}.
\end{proof}

Now that we have showed that any magnetic field $B\in\BB$ yields a unique strong solution of the initial value problem \eqref{VPSU}, we can define an operator mapping every admissible field onto its induced state.

\begin{defn}
\label{CSO}
The operator
\begin{align*}
f. \;\colon \BB \to C([0,T];L^2(\RR^6)),\; B \mapsto f_B
\end{align*}
is called the \textbf{field-state operator}. At this point $f_B$ denotes the unique strong solution of \eqref{VPSU} that is induced by the field $B\in\BB$.
\end{defn}

From now on the notation $f_B$ is to be understood as the value of the field-state operator at point $B\in\BB$.

\section[Continuity and compactness of the field-state operator]{Continuity and compactness of the field-state operator}

Obviously the Lipschitz estimates of Lemma \ref{LIP} hold true for the strong solutions by approximation. 

\begin{cor}
	\label{WLIP}
	Let $L_1,\,L_2,\,L_3,\,c_2,\,c_4,\,c_6$ be the constants from Lemma \ref{ES1} and Lemma~\ref{LIP}. Then for all $B,H\in\BB$,\vspace{-2mm}
	\begin{gather*}
	\begin{aligned}
	\|f_B-f_H\|_{C([0,T];C_b)} &\;\le\; L_1 \|B-H\|_\WW\;,\\
	\|\partial_z f_B-\partial_z f_H\|_{C([0,T];C_b)} &\;\le\; L_2 \|B-H\|_\WW^{\gamma}\;,\\
	\|\partial_t f_B-\partial_t f_H\|_{L^2(0,T;C_b)} &\;\le\; L_3 \|B-H\|_\WW^{\gamma}\;,
	\end{aligned}\\
	\|\delz f_B\|_{C([0,T];C_b)} \le c_2,\;\; \|\delt f_B\|_{L^2(0,T;C_b)} \le c_4,\;\; \|D_z^2 f_B\|_{L^\infty(0,T;W^{2,\beta})} \le c_6.
	\end{gather*}
\end{cor}

The proof of this Corollary is obvious. I states that the field-state operator is globally Lipschitz-continuous with respect to the norm on $C\big([0,T];C_b\big)$ and globally Hölder-continuous with exponent $\gamma=\gamma(\beta)$ with respect to the norm on $W^{1,2}\big(0,T;C_b\big)$ and the norm on $C\big([0,T];C^1_b\big)$. \pskip

The following proposition provides (weak) compactness of the field-state operator that will be very useful in terms of variational calculus.

\begin{prop}
	\label{COMP}
	Let $(B_k)_{k\in\NN} \subset \BB$ be a sequence that is converging weakly in $\WW$ to some limit $B\in\BB$. Then it has a subsequence $(B_{k_j})$ of $(B_k)$ such that
	\begin{align*}
	f_{B_{k_j}} &\rightharpoonup f_B \quad\text{in}\quad W^{1,2}(0,T;L^p)\cap L^2(0,T;W^{1,p}) \cap L^2(0,T;W^{2,\beta}),\; 1\le p <\infty,\\
	f_{B_{k_j}} &\to f_B \quad\text{in}\quad L^2([0,T]\times\RR^6)
	\end{align*}
	if $j$ tends to infinity.
\end{prop}

\begin{proof} Suppose that $(B_k)_{k\in\NN} \subset \BB$ and $B\in\BB$ such that $B_k \rightharpoonup B$ in $\WW$. By Theorem~\ref{GWS}, $f_k:=f_{B_k}$ is bounded in $W^{1,2}(0,T;L^p) \cap L^2(0,T;W^{1,p}\cap W^{2,\beta})$ for every $1\le p \le\infty$. Note that this bound can be chosen independent of $p$. Hence the Banach-Alaoglu theorem and Cantor's diagonal argument \linebreak imply that $(f_k)$ is converging weakly in $W^{1,2}(0,T;L^m) \cap L^2(0,T;W^{1,m}\cap W^{2,\beta})$, for every integer $m \ge 2$, up to a subsequence. Thus there exists some function \linebreak $f\in W^{1,2}(0,T;L^m) \cap L^2(0,T;W^{1,m}\cap W^{2,\beta})$ for every integer $m\ge 2$ such that
	\begin{align*}
	f_k \rightharpoonup f \;\;\text{in}\;\; W^{1,2}(0,T;L^m)\cap L^2(0,T;W^{1,m}) \cap L^2(0,T;W^{2,\beta}),\;\; m\in\NN, m\ge 2.
	\end{align*}
	Thus, by interpolation, $f\in W^{1,2}(0,T;L^p)\cap L^2(0,T;W^{1,p}) \cap L^2(0,T;W^{2,\beta})$ for every $2\le p <\infty$. We will now show that $f$ is a strong solution to the field $B$ by verifying the conditions from Definition \ref{WKS}. \pskip
	
	\textit{Condition }(iv): Let $\eps>0$ be arbitrary. We will now assume that there exists some measurable set $M\subset [0,T]\times\big( \RR^6\setminus \BR \big)$ with Lebesgue-measure $\lambda(M)>0$ such that $f>\eps$ almost everywhere on $M$. Then
	\begin{align*}
	0 < \eps \lambda(M) < \int\limits_M f \;\mathrm d(t,z) = \int\limits_M f-f_k \;\mathrm d(t,z) = \int (f-f_k) \mathds{1}_{M} \;\mathrm d(t,z) \to 0
	\end{align*}
	as $k\to\infty$ which is a contradiction. The case $f<-\eps$ can be treated analogously. Hence $-\eps<f<\eps$ almost everywhere on $[0,T]\times\big( \RR^6\setminus \BR \big)$ which immediately yields $f=0$ almost everywhere on $[0,T]\times\big( \RR^6\setminus \BR \big)$ because $\eps$ was arbitrary. Since $W^{1,2}(0,T;L^p)$ is continuously embedded in $C([0,T];L^p)$ by Sobolev's embedding theorem, we have $\supp f(t) \subset \BR$ even for all $t\in[0,T]$. \pskip
	
	\textit{Condition }(i): The fact that $\supp f(t) \subset B_R(0)$ for all $t\in[0,T]$ directly implies that $f\in W^{1,2}(0,T;L^p)\cap L^2(0,T;W^{1,p})$ for every $1\le p <\infty$ by interpolation. Then we can easily conclude that
	\begin{align*}
	f_k \rightharpoonup f \quad\text{in}\quad W^{1,2}(0,T;L^p)\cap L^2(0,T;W^{1,p})\cap L^2(0,T;W^{2,\beta}), \; 1\le p < \infty\,.
	\end{align*}
	The inequality $\|f\|_{W^{1,2}(0,T;L^p(\RR^6))} + \|f\|_{L^2(0,T;W^{1,p})} \le C$
	where $C>0$ depends only on $\mathring f$, $T$, $K$ and $\beta$ follows directly from the weak convergence and the weak lower semicontinuity of the norm. Since $C$ does not depend on $p$ this inequality holds true for $p=\infty$.\pskip
	
	\textit{Condition }(iii): It holds that $f_k \rightharpoonup f$ in $W^{1,2}(0,T;L^2)$ with $f_k(0) = \mathring f$ almost everywhere on $\RR^6$ for all $k\in\NN$. By Mazur's lemma we can construct some sequence $(f_k^*)_{k\in\NN}$ such that $f_k^*\to f$ in $W^{1,2}(0,T;L^2)$ where for any $k\in\NN$, $f_k^*$ is a convex combination of $f_1,...,f_k$. Then of course $f_k^*(0)= \mathring f$ almost everywhere on $\RR^6$ as well and hence
	\vspace{-0.1cm}
	\begin{align*}
	\|f(0)-\mathring f\|_{L^2} &= \|f(0)-f_k^*(0)\|_{L^2} \le C\; \|f-f_k^*\|_{W^{1,2}(0,T;L^2)} \to 0, \quad k\to\infty\;.\\[-0.75cm]
	\end{align*}
	Thus $f(0)=\mathring f$ almost everywhere on $\RR^6$.\pskip
	
	\textit{Condition }(ii): We know that $f_k\rightharpoonup f$ in $W^{1,2}(0,T;L^2)\cap L^2(0,T;W^{1,2})$ that is $ H^1(]0,T[\times\RR^6)$ due to Lemma \ref{SOBBOCH}. Then, because of the compact support, the Rellich-Kondrachov theorem implies that $f_k\to f$ in $L^2([0,T]\times\RR^6)$, up to a subsequence. From Lemma \ref{NPOT}\,(b) we can conclude that for any $t\in[0,T]$,
	\begin{align*}
	&\|\delx\psi_{f}(t) - \delx\psi_{f_k}(t)\|_{L^2(\BR)} \le C\; \|f(t)-f_k(t)\|_{L^2} \to 0,\quad k\to\infty\;.
	\end{align*}
	For brevity, we will now use the notation
	\begin{align*}
	\mathbf V( \varphi,f,B) := \delt  \varphi + v\cdot\delx  \varphi - \delx\psi_f\cdot \delv  \varphi + (v\times B)\cdot\delv \varphi\;.
	\end{align*}
	Let $ \varphi\in C_c^\infty(]0,T[\times \RR^6))$ be an arbitrary test function. Then $\mathbf V( \varphi,f_k,B_k)$ is bounded in $L^2(]0,T[\times\BR)$ uniformely in $k$ (the bound may depend on $\varphi$). It also holds that
	$$\mathbf V( \varphi,f,B_k) - \mathbf V( \varphi,f_k,B_k) \to 0, \quad k\to \infty \quad\text{in}\; L^2(]0,T[\times\BR)$$ 
	since $\psi_{f_k}\to\psi_f$ in $L^2\big(]0,T[\times \BR\big)$. Moreover, 
	$$ \mathbf V( \varphi,f,B) -\mathbf V( \varphi,f,B_k) \wto 0,\quad k\to \infty \quad\text{in}\; L^2(]0,T[\times\BR).$$
	Hence by integration by parts,\vspace{-2mm}
	\begin{align*}
	&\int\limits_0^T\int \mathbf V(f,f,B)\; \varphi \dz\,\mathrm dt   = \int\limits_0^T\int f\;\mathbf V( \varphi,f,B) - f_k\;\mathbf V( \varphi,f_k,B_k) \dz\,\mathrm dt  \\
	&\; \le \int\limits_0^T\int f\;\Big(\mathbf V( \varphi,f,B) - \mathbf V( \varphi,f,B_k)\Big) \dz\,\mathrm dt 
	+ \int\limits_0^T\int (f-f_k)\;\mathbf V( \varphi,f_k,B_k) \dz\,\mathrm dt \\
	&\qquad + \int\limits_0^T\int f\;\Big(\mathbf V( \varphi,f,B_k) - \mathbf V( \varphi,f_k,B_k)\Big) \dz\,\mathrm dt \\[2mm]
	&\; \to 0, k\to\infty\;.\\[-0.75cm]
	\end{align*}
	As $\varphi$ was arbitrary this implies that $\mathbf V(f,f,B) = 0$ almost everywhere on $[0,T]\times\RR^6$ that is (ii).\pskip
	
	Consequently $f$ is a strong solution to the field $B$ and thus $f=f_B$ because of uniqueness. Furthermore we have showed that there exists a subsequence $(B_{k_j})$ of $(B_k)$ such that $(f_{B_{k_j}})$ is converging in the demanded fashion. \end{proof}

\section{An optimal control problem with a tracking type cost functional}

Again, let $\mathring f \in C^2_c(\RR^6)$ be any given initial datum and let $T>0$ denote some fixed final time. The aim is to control the time evolution of the distribution function in such a way that its value at time $T$ matches a desired distribution function $f_d\in C^2_c(\RR^6)$ as closely as possible. More precisely we want to find a magnetic field $B$ such that the $L^2$-difference $\|f_B(T)-f_d\|_{L^2}$ becomes as small as possible. Therefore we intend to minimize the quadratic cost functional:
\begin{align}
\label{OP1}
\begin{aligned}
&\text{Minimize} \quad J(B) = \frac 1 2 \|f_B(T)-f_d\|_{L^2(\RR^6)}^2 + \frac \lambda 2  \|D_x B\|_{L^2([0,T]\times\RR^3;\RR^{3\times 3})}^2, \\&\text{s.t.}\; B\in\BB.
\end{aligned}
\end{align}
where $\lambda$ is a nonnegative parameter. The field $B$ is the control in this model. As the state $f_B(t)$ preserves the $p$-norm, i.e., $\|f_B(t)\|_{p} = \|\mathring f\|_{p}$ for all $1\le p\le \infty$, $t\in[0,T]$, it makes sense to assume that $\|f_d\|_{p} = \|\mathring f\|_{p}$ for all $1\le p\le \infty$ because otherwise the exact matching $f(T)=f_d$ would be foredoomed to fail. \pskip

At first appearance the term ${\frac \lambda 2 \|D_x B\|_{L^2}^2}$ seems to be useless or even counterproductive as we actually want to minimize the expression ${\|f(T)-f_d\|_{L^2}}$. However, in optimal control theory, such a term is usually added because of its smoothing effect on the control. If $\lambda>0$ a magnetic field is punished by high values of the cost functional if its derivatives become large. Of course the weight of punishment depends on the size of $\lambda$. For that reason the additional term is referred to as the regularization term. Note that the regularity $B\in\HH$ is now necessary to avoid infinite values of the cost functional.\pskip

%

Of course such an optimization problem does only make sense if there actually exists at least one globally optimal solution. This fact will be established in the next Theorem. The proof is quite short as most of the work has already been done in the previous sections.
\begin{thm}
	\label{EXOC1}
	The optimization problem \eqref{OP1} possesses a (globally) optimal solution $\B*$, i.e., for all $B\in\BB$, $J(\B*)\le J(B)$. 
\end{thm}

\begin{proof} Suppose that $\lambda>0$ (if $\lambda = 0$ the proof is similar but even easier). The cost functional $J$ is bounded from below since $J(B)\ge 0$ for all $B\in\BB$. Hence ${M:={\inf}_{B\in\BB} J(B)}$ exists and we can choose a minimizing sequence $(B_k)_{k\in\NN}$ such that $J(B_k) \to M$ if ${k\to\infty}$. Without loss of generality we can assume that $J(B_k)\le M+1$ for all $k\in\NN$. As $\BB\subset\VV$ is weakly compact according to Lemma~\ref{LAC} it holds that $B_k\rightharpoonup \B*$ in $L^2(0,T;W^{2,\beta})\cap L^2(0,T;H^1)$ for some weak limit $\B*\in\BB$ after extraction of a subsequence. Then we know from Proposition~\ref{COMP} that $f_{B_k} \wto f_\B* $ in $W^{1,2}(0,T;L^2)$ after subsequence extraction. By the fundamental theorem of calculus this implies that $f_{B_k}(T)\wto f_{\B*}(T)$ in $L^2(\RR^6)$. Together with the weak lower semicontinuity of the $L^2$-norm this yields
	\begin{align*}
	J(\B*) & \le \underset{k\to\infty}{\lim\inf}\left[ \frac 1 2 \|f_{B_k}(T)-f_d\|_{L^2}^2 \right] + \underset{k\to\infty}{\lim\inf} \left[ \frac \lambda 2  \|D_x B_k\|_{L^2}^2 \right]\\
	& \le \underset{k\to\infty}{\liminf}\left[ \frac 1 2 \|f_{B_k}(T)-f_d\|_{L^2}^2 +  \frac \lambda 2 \|D_x B_k\|_{L^2}^2 \right]
	= \underset{k\to\infty}{\lim}\; J(B_k) = M.
	\end{align*}
	By the definition of infimum this proves $J(\B*) = M$. \end{proof}

Of course this theorem does not provide uniqueness of a globally optimal solution. In general, the optimization problem may have more than one globally optimal solution and, of course, it may have more than one locally optimal solution. To deduce necessary conditions for local optimality it suggests itself to consider the Fréchet derivative of the cost functional $J$. Therefore Fréchet differentiability of the field-state operator must be established, then Fréchet differentiability of the cost functional follows by chain rule. If $J$ is even twice continuously differentiable this can be used to analyze sufficient conditions for local optimality.

\section*{Appendix}

\begin{proof}[Proof of Lemma \ref{SOBBOCH}]
	\textit{Item} (a): According to K. Yosida, $u\in L^p(0,T;W^{k,q}(\RR^d))$ can be approximated by a sequence of finitely valued functions in the following sense: For any ${k\in\NN}$, there exist $\zeta_i^j \in W^{k,q}(\RR^d)$ for $i=1,...k$ and a family of pairwise disjoint open subsets ${I_i^j\subset [0,T]}, {i=1,...,k}$ with $\lambda\big([0,T]\setminus\bigcup_{i=1}^j I_i^j\big) =0$ such that the sequence defined by
	\begin{align*}
	u_j^*(t,x) := \sum_{i=1}^j \mathds{1}_{I_i^j}(t)\; \zeta_i^j (x),\quad t\in[0,T],x\in\RR^d
	\end{align*}
	satisfies $\|u_j^{*}(t)-u(t)\|_{W^{k,q}} \to 0$ for almost every $t\in[0,T]$ as $k\to\infty$. Approximating $x\mapsto \zeta_i^j (x)$ by $C^\infty_c(\RR^d)$-functions and then approximating $t\mapsto \mathds{1}_{I_i^j}(t)$ suitably by $C^\infty(]0,T[)$-functions yields (a). \pskip
	
	\textit{Item} (b): Note that the Meyers-Serrin theorem holds true for Banach-space-valued Sobolev spaces (cf.\;M.\,Kreuter \cite[s.\,4.2]{kreuter}). Hence, for any $j\in\NN$, we can find some function $v_j\in C^\infty\big(]0,T[;L^q(\RR^d)\big)\cap W^{k,p}\big(0,T;L^q(\RR^d)\big)$ with \linebreak$\|v_j - u\|_{W^{k,p}(0,T;L^q)} \le 1/j$. Now, by Friedrich's mollification, we can approximate $v_j$ by a smooth function $u_j\in C^\infty(]0,T[\times\RR^d)$ with $\supp u_j(t) \subset B_{r_j}(0)$ for some radius $r_j>0$ such that $ \|u_j - v_j\|_{W^{k,p}(0,T;L^q)} \le 1/j$. Then $ u_j \to u$ in ${W^{k,p}(0,T;L^q)}$ if $j\to\infty$ which proves (b). \pskip
	
	\textit{Item} (c): Let $u\in L^p(]0,T[\times\RR^d)$ and $v\in L^q(\RR^d)$ be arbitrary where $q:=\tfrac{p}{p-1}$ if $p>1$ and $q=\infty$ if $p=1$. Then, by Fubini's theorem, the function 
	\begin{align*}
	]0,T[\ni t \mapsto \int u(t,x)\, v(x) \dx
	\end{align*}
	is measurable. As $v$ was arbitrary this implies that $t\mapsto u(t)$ is weakly measurable in the Banach space $L^p(\RR^d)$. Since $L^p(\RR^d)$ is separable, we can conclude that $t\mapsto u(t)$ is also strongly measurable in $L^p(\RR^d)$ (cf. B. J. Pettis~\cite{pettis}). Thus $t\mapsto \|u(t)\|_{L^p}$ is measurable with
	\begin{align*}
	\int\limits_0^T \|u(t)\|_{L^p}^p \dt = \|u\|_{L^p(]0,T[\times\RR^d)}^p <\infty,
	\end{align*}
	i.e., $u\in L^p(0,T;L^p(\RR^d))$. Hence $L^p(]0,T[\times\RR^d)\subset L^p(0,T;L^p(\RR^d))$. \pskip
	
	Let now $u\in L^p(0,T;L^p(\RR^d))$ be arbitrary. Then, according to (a), $u$ can be approximated by a sequence $(u_k)\subset C^\infty(]0,T[\subset \RR^d)$ satisfying the support condition. As $(u_k)$ is a Cauchy sequence in $L^p(0,T;L^p(\RR^d))$ it is obviously also a Cauchy sequence in $L^p(]0,T[\times\RR^d)$. Thus $(u_k)$ converges in $L^p(]0,T[\times\RR^d)$ to some function $u^*\in L^p(]0,T[\times\RR^d)$. Because of uniqueness, $u= u^*$ and hence $L^p(0,T;L^p(\RR^d)) \subset L^p(0,T;L^p(\RR^d))$. This proves (c).  \pskip
	
	\textit{Item} (d): Let $\delt u$ and $\delx u = (\partial_{x_1} u,...,\partial_{x_d} u)^T$ denote the partial derivatives of $u\in W^{1,p}(]0,T[\times \RR^d)$, let $\dot u$ denote the derivative of $u\in W^{1,p}(0,T;L^p(\RR^d))$ and let $\grad u$ denote the derivative of $u\in L^p(0,T;W^{1,p}(\RR^d))$. 
	
	At first we will prove the inclusion
	\begin{align}
	\label{INC1}
	W^{1,p}(0,T;L^p(\RR^d)) \cap L^p(0,T;W^{1,p}(\RR^d)) \subset W^{1,p}(]0,T[\times \RR^d)
	\end{align}
	Therefore, let $u \in W^{1,p}(0,T;L^p(\RR^d)) \cap L^p(0,T;W^{1,p}(\RR^d))$ be arbitrary. Then $u$, $\dot u$ and $\grad u$ are in $L^p(]0,T[\times\RR^d)$ because of (c). To show that $u \in W^{1,p}(]0,T[\times \RR^d)$ with $\delt u = \dot u$ let $\phi \in C_c^\infty (]0,T[\times \RR^d)$ be an arbitrary test function. Without loss of generality we can assume that $\phi = \varphi \psi$ with $\varphi\in C^\infty_c(]0,T[)$ and $\psi\in C^\infty_c(\RR^d)$. Then
	\begin{align*}
	&\int\limits_0^T \int u(t,x)\, \delt\phi(t,x) \dx\dt = \int \int\limits_0^T  u(t,x) \dot \varphi(t) \dt\; \psi(x) \dx \\
	&\quad = \int \int\limits_0^T  \dot u(t,x)  \varphi(t) \dt\; \psi(x) \dx = \int\limits_0^T \int \dot u(t,x)\, \phi(t,x) \dx\dt.
	\end{align*}
	This means that $u$ is partially weakly differentiable with respect to $t$ and its weak derivative is $\delt u = \dot u$. It can be proved analogously that $u$ is partially weakly differentiable with respect to $x_i$, $i=1,...,d$ with $\delxi u = [\grad u]_i$. Thus $u \in W^{1,p}(]0,T[\times \RR^d)$ and then, as $u$ was arbitrary, the inclusion \eqref{INC1} follows. \pskip
	
	The proof of $W^{1,p}(]0,T[\times \RR^d) \subset W^{1,p}(0,T;L^p(\RR^d)) \cap L^p(0,T;W^{1,p}(\RR^d)) $ proceeds similarly.
\end{proof}

\bigskip

\end{document}